\documentclass[english,british]{amsart}
\usepackage[T1]{fontenc}
\usepackage[latin9]{inputenc}
\usepackage[a4paper]{geometry}
\usepackage{amsthm}
\usepackage{amssymb}
\usepackage{setspace}
\makeatletter
\numberwithin{equation}{section}
\numberwithin{figure}{section}
\theoremstyle{plain}
\newtheorem{thm}{\protect\theoremname}[section]
  \theoremstyle{plain}
  \newtheorem{cor}[thm]{\protect\corollaryname}
  \theoremstyle{plain}
  \newtheorem{lem}[thm]{\protect\lemmaname}
  \theoremstyle{definition}
  \newtheorem{defn}[thm]{\protect\definitionname}
  \newtheorem{example}[thm]{\protect\examplename}
  \theoremstyle{remark}
  \newtheorem*{rem*}{\protect\remarkname}
  \theoremstyle{plain}
  \newtheorem{prop}[thm]{\protect\propositionname}
  \theoremstyle{remark}
  \newtheorem{rem}[thm]{\protect\remarkname}
\usepackage{amscd}
\usepackage{mathptmx}
\usepackage{bbm}
\usepackage{paralist}

\newcommand{\e}{\mathrm{e}}
\newcommand{\1}{1}
\newcommand{\N}{\mathbb{N}}

\newcommand{\R}{\mathbb{R}}
\newcommand{\C}{\mathbb{C}}

\renewcommand{\Pi}{\pi}
\renewcommand{\emptyset}{\varnothing}
\renewcommand{\hat}{\widehat}

\newcommand\diam{\mathrm{diam}}

\newcommand\id{\mathrm{id}}
\newcommand\Chat{\hat{\mathbb{C}}}
 
\newcommand\T{\mathrm{\mathcal{C}}}
\DeclareMathOperator*{\Span}{span}

\DeclareMathOperator*{\card}{card}

\renewcommand{\vec}[1]{\mathbf{#1}}

\DeclareMathOperator*{\Aut}{Aut}

\DeclareMathOperator*{\Hol}{\textnormal{H\"ol}}
\newcommand{\M}{\mathit{Q}}

\DeclareMathOperator*{\Rat}{Rat}

\usepackage{hyperref}
\@ifundefined{textmu}
 {\usepackage{textcomp}}{}

\makeatother

\usepackage{babel}
  \addto\captionsbritish{\renewcommand{\corollaryname}{Corollary}}
  \addto\captionsbritish{\renewcommand{\definitionname}{Definition}}
  \addto\captionsbritish{\renewcommand{\lemmaname}{Lemma}}
  \addto\captionsbritish{\renewcommand{\propositionname}{Proposition}}
  \addto\captionsbritish{\renewcommand{\remarkname}{Remark}}
  \addto\captionsbritish{\renewcommand{\theoremname}{Theorem}}
  \addto\captionsenglish{\renewcommand{\corollaryname}{Corollary}}
  \addto\captionsenglish{\renewcommand{\definitionname}{Definition}}
  \addto\captionsenglish{\renewcommand{\lemmaname}{Lemma}}
  \addto\captionsenglish{\renewcommand{\propositionname}{Proposition}}
  \addto\captionsenglish{\renewcommand{\remarkname}{Remark}}
  \addto\captionsenglish{\renewcommand{\theoremname}{Theorem}}
  \providecommand{\corollaryname}{Corollary}
  \providecommand{\definitionname}{Definition}
  \providecommand{\lemmaname}{Lemma}
  \providecommand{\propositionname}{Proposition}
  \providecommand{\remarkname}{Remark}
\providecommand{\theoremname}{Theorem}
  \providecommand{\examplename}{Example}
\begin{document}

\title{Pointwise H\"older Exponents of the Complex Analogues \\ of the Takagi Function in Random Complex Dynamics}

\author{Johannes Jaerisch and Hiroki Sumi}

\date{17th May 2017. Published in Adv. Math. 313 (2017) 839-874. \ {\em MSC2010}: 37H10, 37F15.}

\thanks{\ \newline \noindent Johannes Jaerisch \newline  Department of Mathematics, 
Faculty of Science and Engineering, Shimane University, Nishikawatsu 1060 
 Matsue, Shimane
 690-8504, Japan E-mail: jaerisch@riko.shimane-u.ac.jp\ \ Web: http://www.math.shimane-u.ac.jp/$\sim$jaerisch/ %\newline Tel: \ \ Fax:  
 \newline \newline   \noindent Hiroki Sumi  (corresponding author)  \newline 
 Course of Mathematical Science, 
 Department of Human Coexistence, Graduate School of Human and Environmental 
 Studies, Kyoto University, Yoshida-nihonmatsu-cho, Sakyo-ku, Kyoto,  
 606-8501, Japan\newline E-mail: sumi@math.h.kyoto-u.ac.jp\ \ Web: 
http://www.math.h.kyoto-u.ac.jp/$\sim$sumi/index.html
%http://www.math.sci.osaka-u.ac.jp/$\sim$sumi/
 %\newline Tel: +81-(0)6-6850-5307\ \ Fax: +81-(0)6-6850-5327
 }

\vspace{-7mm} 
\begin{abstract}
We consider hyperbolic random complex dynamical systems on the Riemann sphere 
with separating condition and multiple minimal sets. 
%and 
%the dynamics of assoiciated semigroups of rational maps. 
We investigate the H\"older
regularity of the function $T$ of the probability of tending to one minimal set, 
the partial derivatives of $T$ with respect to the probability parameters, which can be regarded as complex analogues of the Takagi function, and the 
higher partial derivatives $C$ of $T.$   
Our main result gives a dynamical description of the pointwise
H\"older exponents of $T$ and $C$, which allows us to determine the spectrum of pointwise
H\"older exponents by employing the multifractal formalism in ergodic
theory. Also, we prove that the bottom 
of the spectrum $\alpha _{-}$ is strictly less than $1$, 
which allows us to show that the averaged system acts chaotically on the Banach space 
$C^{\alpha }$ of 
$\alpha $-$\Hol$der continuous functions for every $\alpha \in (\alpha _{-},1)$, 
though the averaged system behaves very mildly (e.g. we have spectral gaps)  on 
$C^{\beta }$ for small $\beta >0.$   
\end{abstract}

\maketitle
\vspace{-8mm} 
\section{Introduction and Statement of Results}
\label{Introduction}

In this paper, we consider random dynamical systems of  rational maps on the Riemann sphere 
$\Chat $. The study of random complex dynamics was initiated by J.E. Fornaess and N. Sibony 
(\cite{MR1145616}). 
There  are many new interesting phenomena in random dynamical systems, so called 
randomness-induced phenomena or noise-induced phenomena, which cannot hold in 
the deterministic iteration dynamics. For the motivations and recent research of random complex dynamical systems 
focused on the randomness-induced phenomena, see the second author's works 
\cite{s11random, S13Coop, rcddc, twogenerator}. 
In these papers it was shown that for a generic i.i.d. random dynamical system of complex polynomials 
of degree two or more,  
the system acts very mildly on the space of continuous functions on $\Chat$ 
and on the space $C^{\alpha }(\Chat )$ for small $\alpha \in (0,1)$,  
 where $C^{\alpha }(\Chat )$ denotes the 
Banach space of $\alpha $-H\"older continuous functions on $\Chat $ 
endowed with $\alpha $-H\"older norm, 
but under certain conditions the system still acts chaotically  
on the space $C^{\beta }(\Chat )$ for some $\beta \in (0,1)$ close to $1$.    
Thus, we investigate the gradation between chaos and order in random (complex) dynamical systems.   

In order to show the main ideas of the paper, let Rat denote the set of all non-constant rational maps on 
$\Chat. $ This is a semigroup whose semigroup operation is the composition of maps. 
Throughout the paper, 
let $s\ge1$ and let $(f_{1},\ldots, f_{s+1})\in (\mbox{Rat})^{s+1}$ with $\deg(f_{i})\geq 2,i=1,\ldots, s+1.$   
Let  $\vec{p}=(p_{1},\dots,p_{s})\in(0,1)^{s}$
with $\sum_{i=1}^{s}p_{i}<1$ and let  $p_{s+1}:= 1-\sum_{i=1}^{s}p_{i}$. 
We consider the (i.i.d.) random dynamical system on $\Chat$ such that 
at every step we choose $f_{i}$ with probability $p_{i}.$ 
This defines a Markov chain with state space $\Chat $ such that 
 for each $x\in \Chat $ and for each Borel measurable subset $A$ of $\Chat$, 
the transition probability $p(x,A)$ from $x$ to $A$  is equal to $\sum _{i=1}^{s+1}p_{i}1_{A}(f_{i}(x))$, where 
$1_{A}$ denotes the characteristic function of $A$.  
Let $G=\left\langle f_{1},\dots,f_{s},f_{s+1}\right\rangle $
be the rational semigroup (i.e., subsemigroup of Rat) generated by 
$\{f_{1},\ldots, f_{s+1}\}.$ More precisely, 
$G=\{ f_{\omega _{n}}\circ \cdots \circ f_{\omega _{1}}: 
n\in \Bbb{N}, \omega _{1},\ldots, \omega _{n}\in \{ 1,\ldots, s+1\} \}.$ 
We denote by $F(G)$ the 
maximal open subset of $\Chat $ 
on which $G$ is equicontinuous with respect to the spherical distance on $\Chat .$ 
%set of initial values $z\in \Chat$ 
%for which there is a neighborhood 
%$U$ of $z$ such that $\{ g|_{U}:U\rightarrow \Chat\}_{g\in G}$ is equicontinuous on $U$ with respect to 
%the spherical distance on $\Chat. $ 
The set $F(G)$ is called the Fatou set of $G$, and the set $J(G):=\Chat \setminus F(G)$ is called the Julia set of $G$. 
%We have $J(G)=\cup _{i=1}^{s+1}f_{i}^{-1}(J(G))$ 
We remark that in order to investigate random complex dynamical systems, 
it is very important to investigate the dynamics of associated rational semigroups. 
The first study of dynamics of rational semigroups was conducted by A. Hinkkanen and G. J. 
Martin (\cite{MR1397693}), who were interested in the role of polynomial semigroups (i.e., 
semigroups of non-constant polynomial maps) while studying various one-complex-dimensional 
moduli spaces for discrete groups, and by F. Ren's group (\cite{MR1435167}), 
who studied such semigroups from the perspective of random dynamical systems. 
For the interplay of random complex dynamics and dynamics of rational semigroups, 
see \cite{MR1767945}--\cite{twogenerator}, \cite{sspb, su09, su13, JS13b, JStams}. \\ 
%\cite{s11random, S13Coop, twogenerator, su13, JS13b}. \\ 
\vspace{-1.2mm} 
Throughout the paper, we assume the following. 
\vspace{-2.5mm} 
\begin{itemize}
\item[(1)] $G=\langle f_{1},\ldots, f_{s+1}\rangle $ is hyperbolic, i.e., 
we have $P(G)\subset F(G)$, where 
$$P(G):=\overline{\bigcup _{g\in G\cup \{ \id\} }g(\cup _{i=1}^{s+1}\{ \mbox{critical values of }
f_{i}:\Chat \rightarrow \Chat \} )}.
\mbox{ Here, the closure is taken in }\Chat .$$ 
\item[(2)] $(f_{1},\ldots, f_{s+1})$ satisfies the separating condition, i.e., 
$f_{i}^{-1}(J(G))\cap f_{j}^{-1}(J(G))=\emptyset $ whenever $i,j\in \{1,\ldots, s+1\}, i\neq j.$ 
\item[(3)] There exist at least two minimal sets of $G$. Here, 
a non-empty compact subset $K$ of $\Chat $ is called a minimal set of $G$ 
if $K=\overline{\cup _{g\in G}\{ g(z)\} }$ for each $z\in K.$   
\end{itemize}  
\vspace{-2.3mm} 
Note that by assumption (2), \cite[Lemma 1.1.4]{sumihyp} and \cite[Theorem 3.15]{s11random}, 
we have that 
there exist at most finitely many 
minimal sets of $G$.  Moreover, 
denoting by $S_{G}$ the union of minimal sets of $G$ and setting $I:=\{1,\dots, s+1\}$, we have that 
for each $z\in \Chat$ there exists a Borel subset $A_{z}$ of 
$I^\N$ with $\tilde{\rho }_{\vec{p}}(A_{z})=1$  such that $d(f_{\omega _{n}}\cdots f_{\omega _{1}}(z),S_{G})\rightarrow 0$ as $n\rightarrow \infty$ 
for all $\omega =(\omega _{i})_{i=1}^{\infty }\in A_{z}$, where  $\tilde{\rho }_{\vec{p}}:=\otimes _{n=1}^{\infty }\rho _{\vec{p}}$ denotes the product measure on  $I^\N$ given by  $\rho _{\vec{p}}:=\sum _{i=1}^{s+1}p_{i}\delta _{i}$ with  $\delta _{i}$ denoting  the 
Dirac measure concentrated at $i\in I$.

Throughout,  we fix a minimal set $L$ of $G$ 
(e.g. $L=\{ \infty \} $ when $G$ is a polynomial semigroup). 
Denote by $T_{\vec{p}}(z)$ the probability of tending to $L$ of
the process on $\Chat$ which starts in $z\in\Chat$ and which is
given by drawing independently with probability $p_{i}$ the map $f_{i}$. 
More precisely, 
$T_{\vec{p}}(z):= 
\tilde{\rho }_{\vec{p}}(\{ 
\omega =(\omega _{i})_{i=1}^{\infty }\in I^\N : 
d(f_{\omega _{n}}\circ \cdots \circ f_{\omega _{1}}(z), L)
\rightarrow 0 \mbox{ as }n\rightarrow \infty \} )$. 
%, where 
%$\delta _{i}$ denotes the Dirac measure concentrated at $i\in \{ 1,\ldots, s+1\} .$   
It was shown by the second author in \cite{S13Coop} that, for each
$\vec{p}=(p_{1},\dots,p_{s})$ there exists $\alpha \in\left(0,1\right)$
such that $\vec{x}=(x_{1},\dots,x_{s})\mapsto T_{(x_{1},\dots,x_{s},1-\sum_{i=1}^{s}x_{i})}\in C^{\alpha}(\Chat)$
is real-analytic in a neighbourhood of $\vec{p}$, where $C^{\alpha}(\Chat)$
denotes the $\C$-Banach space of $\alpha$-H\"older continuous $\Bbb{C}$-valued functions 
on $\Chat$ endowed with $\alpha$-H\"older norm $\| \cdot \| _{\alpha}$ (Remark~\ref{rem:motivation}).  
Thus it is very natural and important  to consider the following. 
For $\N_{0}:=\N \cup \{ 0\} $ and $\vec{n}=\left(n_{1},\dots,n_{s}\right)\in\N_{0}^{s}$ we denote
by $C_{\vec{n}}\in C^{\alpha}(\Chat)$ the higher order partial derivative
of $T_{\vec{p}}$ of order $\left|\vec{n}\right|:=\sum_{i=1}^{s}n_{i}$
with respect to the probability parameters given by 
\[
C_{\vec{n}}(z):=\frac{\partial^{\left|\vec{n}\right|}T_{(x_{1},\dots,x_{s},1-\sum_{i=1}^{s}x_{i})}(z)}{\partial x_{1}^{n_{1}}\partial x_{2}^{n_{2}}\cdots\partial x_{s}^{n_{s}}}\Big|_{\vec{x}=\vec{p}},\quad z\in\Chat. 
\]
\foreignlanguage{english}{
These functions are introduced in \cite{S13Coop} by the second author.   
We introduce the $\C$-vector space 
\[
\T:=\Span\left\{ C_{\vec{n}}\mid\vec{n}\in\N_{0}^{s}\right\} \subset C^{a}(\Chat),
\]
which consists of all the finite complex linear combinations of elements from
$\left\{ C_{\vec{n}}\mid\vec{n}\in\N_{0}^{s}\right\} $. }The first
order derivatives are called complex analogues of the Takagi function
in \cite{S13Coop}. Note that $C_{0}=T_{\vec{p}}$. \\ 
\vspace{-2mm} 
For an element $C\in\T$ and $z\in\Chat$ the H\"older exponent \foreignlanguage{english}{$\Hol\left(C,z\right)$}
is given by 
\[
\Hol\left(C,z\right):=\sup\left\{ \alpha\in[0,\infty):\limsup_{y\rightarrow z,y\neq z}\frac{\left|C\left(y\right)-C\left(z\right)\right|}{d\left(y,z\right)^{\alpha}}<\infty\right\} \in\left[0,\infty\right],
\]
where $d$ denotes the spherical distance on $\Chat$. 
It was shown
in \cite{JS13b} that the level sets 
\vspace{-2mm} 
\[
H(C_{0},\alpha):=\{ z\in\Chat:\Hol(C_{0},z)=\alpha \} ,\quad\alpha\in\R,
\]
satisfy the multifractal formalism. In particular, 
there exists an interval of parameters $\left(\alpha_{-},\alpha_{+}\right)$
such that the Hausdorff dimension of $H(C_{0},\alpha)$ is positive
and varies real analytically on $(\alpha _{-}, \alpha _{+})$ (see Theorem \ref{thm:multifractal}
below). 

The first main result of this paper gives a dynamical description of the
pointwise H\"older exponents for an arbitrary $C\in\T$. We say that $C=\sum_{\vec{n}\in\N_{0}^{s}}\beta_{\vec{n}}C_{\vec{n}}\in {\mathcal C}$
is non-trivial if there exists $\vec{n}\in\N_{0}^{s}$ with $\beta_{\vec{n}}\neq0$.
It turns out in Theorem \ref{main-thm} below that every non-trivial
$C\in\T$ has the same pointwise H\"older exponents.
To state the result, we define
the skew product map (associated with $(f_{i})_{i\in I}$) 
(see \cite{MR1767945}) 
\[
\tilde{f}:I^{\N}\times\hat{\C}\rightarrow I^{\N}\times\hat{\C},\quad\tilde{f}\left(\omega,z\right):=\left(\sigma\left(\omega\right),f_{\omega_{1}}\left(z\right)\right),
\]
where $\sigma:I^{\N}\rightarrow I^{\N}$ denotes the shift map
given by $\sigma\left(\omega_{1},\omega_{2},\dots\right):=\left(\omega_{2},\omega_{3},\dots\right)$,
for $\omega=\left(\omega_{1},\omega_{2},\dots\right)\in I^{\N}$. 
For every $\omega =(\omega _{j})_{j\in \N }\in I^{\N}$ and $n\in \N $,  
let $f_{\omega |_{n}}:=f_{\omega _{n}}\circ \cdots \circ f_{\omega _{1}}$ and 
we denote by $F_{\omega }$ the maximal open subset of $\Chat$ 
on which $\{ f_{\omega |_{n}}\}_{n\in \N}$ is equicontinuous with respect to $d.$ Let $J_{\omega }:=\Chat \setminus F_{\omega }.$
The Julia set of $\tilde{f}$ is given by 
$J\left(\tilde{f}\right)=\overline{\cup _{\omega \in I^{\N}}\{ \omega \} \times J_{\omega }}$ 
where the closure is taken in the product space $I^{\N }\times \Chat $. 
Note that denoting by 
$\pi :I^{\N} \times \Chat \rightarrow \Chat $ the canonical projection, 
$\pi :J(\tilde{f})\rightarrow J(G)$ is a homeomorphism (\cite[Lemma 4.5]{s11random}, 
\cite[Lemma 1.1.4]{sumihyp} and 
assumption (2)) and $\pi \circ \tilde{f}=\sigma \circ \pi.$   
We introduce the potentials $\tilde{\varphi}, \tilde{\psi}:J(\tilde{f})\rightarrow\R$
given by 
\[
\tilde{\varphi}\left(\omega,z\right):=-\log\left\Vert f_{\omega_{1}}'\left(z\right)\right\Vert ,\quad\quad\tilde{\psi}(\omega,z):=\log p_{\omega_{1}}, 
\]
where $\| \cdot \| $ denotes the norm of the derivative with respect to the spherical metric on $\Chat.$ 
%where $J\left(\tilde{f}\right)$ refers to the Julia set of $\tilde{f}$.
%there is a bijection between the Julia set $J(G)$ of the
%semigroup $G$ generated by $f_{1},\dots,f_{s+1}$ and $J\left(\tilde{f}\right)$. 
Note that $\tilde{f}^{-1}(J(\tilde{f}))=J(\tilde{f})=\tilde{f}(J(\tilde{f}))$ 
(\cite{MR1767945}). 
We denote by $S_{n}\tilde{u}$ the $n$-th ergodic sum 
$\sum _{j=0}^{n-1}\tilde{u}\circ \tilde{f}^{j}$
of the dynamical system $(J(\tilde{f}), \tilde{f})$ 
with respect to a function $\tilde{u}$ on $J(\tilde{f})$.
\selectlanguage{english}%
\begin{thm}
\label{main-thm} For every non-trivial 
$C=\sum_{\vec{n}\in\N_{0}^{s}}\beta_{\vec{n}}C_{\vec{n}}\in\T$ we
have 
\begin{equation}
\Hol\left(C,z\right)=\liminf_{k\rightarrow\infty}\frac{S_{k}\tilde{\psi}\left(\omega,z\right)}{S_{k}\tilde{\varphi}\left(\omega,z\right)},\quad\mbox{for all }(\omega,z)\in J(\tilde{f}).\label{eq:hoelder-is-quotientofbirkhoff}
\end{equation}

\end{thm}
Combining Theorem \ref{main-thm} with our results from \cite[Theorem 1.2]{JS13b}
on the multifractal formalism, we establish the multifractal formalism
for the pointwise H\"older exponents of an arbitrary non-trivial $C\in\T $.
To state the results, 
%we recall some of the terminology from \cite{JS13b}.
for any non-trivial $C\in \T $ and $\alpha\in\R$ we denote
by 
\[
H\left(C,\alpha\right):=\{ y\in\Chat:\Hol\left(C,y\right)=\alpha\} 
\]
the level set of prescribed \foreignlanguage{british}{H\"older exponent
$\alpha$. \foreignlanguage{british}{The range of the multifractal spectrum is
given by 
\vspace{-1.4mm} 
\[
\alpha_{-}:=\inf\left\{ \alpha\in\R:H\left(C,\alpha\right)\neq\emptyset\right\} \in \R \quad\mbox{and }\quad\alpha_{+}:=\sup\left\{ \alpha\in\R:H\left(C,\alpha\right)\neq\emptyset\right\} \in \R.
\]
}
%\vspace{-1.2mm} 
By Theorem }\ref{main-thm}, 
the sets $H\left(C,\alpha\right)$ coincide for all non-trivial $C\in\T$. Thus,  $\alpha _{-}$ and $\alpha _{+}$ do not depend on the choice of a non-trivial $C\in {\mathcal C}$. Also,   
 $\alpha_->0$ (\cite[Theorem 2.6]{MR1625124}, see also 
 Corollary~\ref{cor:alpha}).  
\begin{thm}[For the detailed statements, see Theorem \ref{thm:mf-for-hoelderexponent}]
%[\cite{JS13b}]  
\label{thm:multifractal}
%For every non-trivial $C\in\T$
All of the following hold. 
\vspace{-2.1mm} 
\begin{itemize}
%\item \label{mf1}
%$C$ is $a$-H\"older continuous on $\Chat$ for every
%$a<\alpha_{-}$. Moreover, $C_{0}$ is $\alpha_{-}$-H\"older continuous on $\Chat$. 
\item[{\em (1)}]  Let $C\in {\mathcal C}$ be non-trivial. If $\alpha_{-}<\alpha_{+}$ then the Hausdorff dimension function
$\alpha\mapsto\dim_{H}\left(H\left(C,\alpha\right)\right)$, $\alpha\in\left(\alpha_{-},\alpha_{+}\right)$,
defines a real analytic and strictly concave positive function on
$\left(\alpha_{-},\alpha_{+}\right)$ with maximum value $\dim_{H}\left(J(G)\right)$.
If $\alpha_{-}=\alpha_{+}$, then we have 
$H\left(C,\alpha_{-}\right)=J(G).$  
\item[{\em (2)}]  We have $\alpha_{-}=\alpha_{+}$ if and only if there exist an automorphism
$\theta\in\Aut\bigl(\Chat\bigr)$, complex numbers $\left(a_{i}\right)_{i\in I}$
and $\lambda\in\R$ such that for all $i\in I$ and $z\in\Chat$,
\[
\theta\circ f_{i}\circ\theta^{-1}\left(z\right)=a_{i}z^{\pm\deg\left(f_{i}\right)}\quad\mbox{and}\quad\log\deg\left(f_{i}\right)=\lambda\log p_{i}.
\]
\end{itemize} 
\end{thm}
\vspace{-0.0mm} 
In the next theorem we determine the actual H\"older class of every non-trivial $C\in \mathcal C$.
\begin{thm}
\label{thm:c0holder}
For every non-trivial $C\in {\mathcal C}$ and for every $\alpha<\alpha _{-}$, 
the function $C$ is $\alpha$-H\"older continuous on $\Chat. $ Moreover, $C_{0}$ 
is $\alpha _{-}$-H\"older continuous on $\Chat .$ 
\end{thm}
\selectlanguage{british}%
%\vspace{-3mm} 

To prove  Theorem~\ref{thm:c0holder} 
we develop some ideas from 
\cite{kessenondifferentiabilityminkowski, MR2576266} for interval maps.  The relation between the H{\"o}lder continuity of singular  measures and their multifractal  spectra has  been first observed in \cite{kessenondifferentiabilityminkowski}, where it was shown that the H{\"o}lder continuity of the  Minkowski's question mark function coincides with the bottom of the Lyapunov spectrum of the Farey map. In \cite{MR2576266} a similar result has been obtained for expanding  interval maps.

In the following Theorem \ref{thm:less-than-one} we prove that $\alpha _{-}<1$. This result allows us to give a complete answer to two important problems raised in \cite{S13Coop}, which greatly improves the previous partial results in \cite{s11random, S13Coop, JS13b}.  The first implication is that,   under the assumptions of our paper, every non-trivial $C\in {\mathcal C}$ is not differentiable at every point of a Borel dense subset $A$ of $J(G)$ with $\dim _{H}(A)>0$. Secondly, we obtain in  Theorem~\ref{thm:alphachaos}  that the averaged system still acts chaotically  
on the space $C^{\alpha }(\Chat )$ for any $\alpha \in (\alpha _{-},1)$, although the averaged system acts very mildly on the Banach space $C(\Chat )$ 
of $\C $-valued continuous functions on $\Chat $ endowed with 
the supremum norm and on the Banach space $C^{\alpha }(\Chat )$ for small $\alpha >0$ (see \cite[Lemma 1.1.4]{sumihyp}, 
\cite[Theorem 3.15]{s11random} and \cite[Theorem 1.10]{S13Coop}). 
%the averaged system still acts chaotically  
%on the space $C^{\alpha }(\Chat )$ for any $\alpha \in (\alpha _{-},1).$ 
%Thus the result ``$\alpha _{-}<1$'' is very important. 
%Note that the assumptions of our paper are much weaker 
%than those of the previous partial results in 
%\cite{s11random, S13Coop, JS13b}.   
We recall  that if $\Hol(C,z)<1$ then $C$ is not differentiable at $z$. If $\Hol(C,z)>1$ then $C$ is differentiable at $z$ and 
the derivative of $C$ at $z$ is zero.
% and \ref{thm:alphachaos} 
%(especially, ``$\alpha _{-}<1$'') indicate that under the assumptions of our paper, 
%every non-trivial $C\in {\mathcal C}$ is not differentiable at every point of a Borel 
%dense subset $A$ of $J(G)$ with $\dim _{H}(A)>0$, and 
\begin{thm}\label{thm:less-than-one}
We  have $\alpha _{-}<1$.
Moreover, for every $\alpha \in (\alpha _{-}, \min\{\alpha_+,1\})$
there exists a Borel dense subset $A$ of $J(G)$
with $\dim _{H}(A)>0$ such that
for every non-trivial $C\in \T$ and for every $z\in A$,
we have $\Hol(C,z)=\alpha <1$ and $C$ is not differentiable at $z.$
\end{thm}

The proof of Theorem \ref{thm:less-than-one} will be postponed to Section \ref{section:pfl-t-o}.  
In the proof, we combine the result that $C_{0}$ is $\alpha _{-}$-$\Hol$der continuous on $\Chat $
 (Theorem~\ref{thm:c0holder}),  
 the multifractal analysis on the pointwise $\Hol$der exponents of $C_{0}$ 
(Theorems~\ref{thm:multifractal} and \ref{thm:mf-for-hoelderexponent}), 
an argument on  Lipschitz functions on $\C $  
and the fact that  $\dim _{H}(J(G))<2$, which follows from our assumptions (1) and (2) (\cite{MR1625124}).  
 
 To state Theorem~\ref{thm:alphachaos}, 
 let $M:C(\Chat )\rightarrow C(\Chat)$ be the transition operator 
 of the system which is defined by 
 $M(\phi )(z)=\sum _{j=1}^{s+1}p_{j}\phi (f_{j}(z))$, 
 where $\phi \in C(\Chat ), z\in \Chat.$ Note that 
 $M(C^{\alpha }(\Chat ))\subset C^{\alpha }(\Chat )$ for any $\alpha \in (0,1].$ 
\begin{thm}
\label{thm:alphachaos}
Let $\alpha \in (\alpha _{-},1)$ and let 
 $\phi \in C^{\alpha }(\Chat )$ such that
$\phi |_{L}=1$ and $\phi|_{L'}=0$ for every  minimal set 
$L' $ of $G$ with $L'\neq L$. Then $\| M^{n}(\phi )\| _{\alpha }
\rightarrow \infty $ as $n\rightarrow \infty .$ In particular, 
for every $\xi \in C^{\alpha }(\Chat )$ and for every $a\in \C \setminus \{  0\} $, 
we have $\| M^{n}(\xi +a\phi )-M^{n}(\xi )\| _{\alpha }\rightarrow \infty $ as 
$n\rightarrow \infty .$ 
\end{thm}

\begin{proof} 
Recall from \cite{s11random} that $C_{0}=\lim_{n\rightarrow \infty}M^{n}(\phi )$ in $C(\Chat )$. Suppose for a contradiction that  there exist a subsequence $(n_{j})$ and a constant $K>0$ such that
$| M^{n_{j}}(\phi )(x)-M^{n_{j}}(\phi )(y)|\leq Kd(x,y)^{\alpha }$ for all $j,x,y.$ 
Letting $j\rightarrow \infty $ we have  $C_{0}\in C^{\alpha }(\Chat )$. But, this would imply that $\alpha_-\ge \alpha$ which is a contradiction.
\end{proof}
%\vspace{-1mm} 
We now present the corollaries of our main results. The first one establishes that every non-trivial $C\in \T$ varies precisely on the Julia set $J(G)$. This follows immediately from Theorem \ref{main-thm} because  the
right-hand side of (\ref{eq:hoelder-is-quotientofbirkhoff}) is  
always finite (\cite[Theorem 2.6]{MR1625124}, see also 
Corollary~\ref{cor:alpha}). This generalizes a previous result from  \cite{s11random} for $C_{0}=T_{\vec{p}}$ and a partial result for the higher order partial derivatives from \cite{S13Coop}.  
% $C\in\T\subset C(\Chat)$ is 
%locally constant on the Fatou set $F(G)$ 
%and  that  $T_{p}=C_{0}$
%varies at any point of the Julia set $J(G)$, which is a fractal set.   
% Since the
%right-hand side of (\ref{eq:hoelder-is-quotientofbirkhoff}) in Theorem
%\ref{main-thm} is always finite, we obtain the following corollary. 
\begin{cor}
\label{cor:varies}
Every non-trivial $C\in\T$ varies precisely on $J(G)$, i.e., $J(G)$ is equal to the set of points $z_{0}\in \Chat$ 
such that $C$ is not constant in any neighborhood of $z_{0}$ 
in $\Chat .$ In particular, 
the functions $C_{\vec{n}}, \vec{n}\in \N_{0}^{s}$, are linearly independent over $\C $ and 
$\T$ has a representation as a direct
sum of vector spaces given by 
\[
\T=\bigoplus_{\vec{n}\in\N_{0}^{s}}\C C_{\vec{n}}.
\]
\end{cor}
We remark again that  $0<\dim _{H}(J(G))<2$  (\cite{MR1625124}). 
  
%\newpage 
By combining  Theorem \ref{main-thm} with Birkhoff's ergodic theorem
we obtain the following extension of \cite[Theorem 3.40 (2)]{S13Coop}. Recall that a Borel probability measure $\nu$ on $J(\tilde{f})$ is called $\tilde{f}$-invariant if $\nu(\tilde{f}^{-1}(A))=\nu(A)$ for every Borel set  $A\subset J(\tilde{f})$. 
\begin{cor}\label{cor:ergodictheorem}
Let $\nu$ be an $\tilde{f}$-invariant ergodic Borel probability  measure on
$J\left(\tilde{f}\right)$. 
Let $\pi:I^{\N}\times \Chat \rightarrow \Chat$ denote the canonical projection onto $\Chat$. 
Then there exists a Borel subset $A$ of $J(G)$ with 
$(\pi _{\ast }(\nu ))(A)=1$ such that 
for every non-trivial $C\in\T$ and for every $z\in A$, we
have 
\[
\Hol\left(C,z\right)=\frac{-\int\log p_{\omega_{1}}d\nu(\omega,x)}{\int\log\Vert f_{\omega_{1}}'(x)\Vert d\nu(\omega,x)},  
%,\quad\mbox{for }\pi_{*}(\nu)-\mbox{almost every }z\in J(\tilde{f}),
\mbox{\ \ where }
\omega=\left(\omega_{1},\omega_{2},\dots\right)\in I^{\N}. 
\] 

\end{cor}
By combining Corollary~\ref{cor:ergodictheorem} with 
\cite[Theorem 3.82]{s11random} in which the potential theory was used, 
we obtain the following result (Corollary~\ref{cor:nondiff}) on 
the pointwise $\Hol$der exponents and the non-differentiability of elements
of $\T$. To state the result, when $G$ is a polynomial semigroup, 
we denote by $\tilde{\mu}_{\vec{p}}$ the maximal relative
entropy measure on $J(\tilde{f})$ for $\tilde{f}$ with respect to $(\sigma,\tilde{\rho}_{\vec{p}}$)  (see \cite{MR1767945}, \cite[Remark 3.79]{s11random}).
Note that $\tilde{\mu }_{\vec{p}}$ is $\tilde{f}$-invariant and ergodic (\cite{MR1767945}).  
 Let $\mu_{\vec{p}}=\pi_{*}(\tilde{\mu}_{\vec{p}})$.
% Moreover, 
For any $(\omega ,z)\in I^{\Bbb{N}}\times \Chat$, 
let ${\mathcal G}_{\omega }(z):=
\lim_{n\rightarrow \infty }(1/\deg (f_{\omega |_{n}}))\log 
^{+}|f_{\omega |_{n}}(z)|$, where 
%$f_{\omega |_{n}}=f_{\omega _{n}}\circ \cdots \circ f_{\omega _{1}}$ and 
$\log ^{+}(a):=\max \{ \log a, 0\} $ for every $a>0.$ 
  By the argument in \cite{Sesskew}, we have that 
  ${\mathcal G}_{\omega }(y)$
   exists for every $(\omega ,z)\in I^{\Bbb{N} }\times \C $,
   $(\omega ,z)\in I^{\Bbb{N}}\times \C \mapsto {\mathcal G}_{\omega }(z)$ is continuous on $I^{\Bbb{N}}\times \C $,  
   ${\mathcal G}_{\omega }$ is  subharmonic  on $\C $ 
   and $ {\mathcal G}_{\omega }$ restricted to 
   the intersection of $\C $ and the basin $A_{\infty ,\omega }$ of $\infty $ for 
   $\{ f_{\omega |_{n}}\} _{n=1}^{\infty }$ is the Green's function 
   on $A_{\infty ,\omega }$ with pole at $\infty .$  
   Let $\Lambda (\omega )=\sum _{c}{\mathcal G}_{\omega }(c)$, 
   where $c$ runs over all critical points of $f_{\omega _{1}}$ in 
   $A_{\infty, \omega } $, counting multiplicities.  
   Note that $\mu _{\vec{p}}=\int _{I^{\N }}dd^{c}{\mathcal G}_{\omega }
   d\tilde{\rho }_{\vec{p}}(\omega )$ where $d^{c}=(\sqrt{-1}/2\pi )
   (\overline{\partial}-\partial)$ (\cite[Lemma 5.51]{s11random}), 
      supp$\,\mu _{\vec{p}}=J(G)$ and $\mu _{\vec{p}}$ is 
non-atomic (\cite{MR1767945}). Also, 
%by the proof of 
%\cite[Theorem 3.82]{s11random}, 
we have 
$\dim _{H}(\mu _{\vec{p}})=$ 
$(\sum _{i\in I}p_{i}\log \deg f_{i}-\sum _{i\in I}p_{i}\log p_{i})
/(\sum _{i\in I}p_{i}\log \deg f_{i}+\int _{I^{\Bbb{N}}}\Lambda (
\omega )d\tilde{\rho }_{\vec{p}}(\omega ))>0$ (\cite[Proof of Theorem 3.82]{s11random}). Here,  
$\dim _{H}(\mu _{\vec{p}}):=
\inf \{ \dim _{H}(A)\} $ where the infimum is taken over all 
 Borel subsets $A$ of  $J(G)$ with $ \mu _{\vec{p}}(A)=1.$    
 \begin{cor}
\label{cor:nondiff}
{\em (1) } Suppose that $f_{1},\ldots, f_{s+1}$ are polynomials. 
Then 
%under the above notations, 
there exists a Borel dense subset $A$ of 
$J(G)$ with $\mu _{\vec{p}}(A)=1$ and 
$\dim _{H}(A)\geq (\sum _{i\in I}p_{i}\log \deg f_{i}-\sum _{i\in I}p_{i}\log p_{i})
/(\sum _{i\in I}p_{i}\log \deg f_{i}+\int _{I^{\Bbb{N}}}\Lambda (
\omega )d\tilde{\rho }_{\vec{p}}(\omega ))$ $>0$ 
  such that 
for every non-trivial $C\in \T$ and for every $z\in A$, we have 
$$ \Hol (C,z)= 
\frac{-\sum _{i\in I}p_{i}\log p_{i}}
{\sum _{i\in I}p_{i}\log \deg f_{i}+\int _{I^{\Bbb{N}}}
\Lambda (\omega )d\tilde{\rho }_{\vec{p}}(\omega )}. 
$$ 
{\em (2)} Suppose that $f_{1},\dots,f_{s+1}$ are polynomials satisfying
at least one of the following conditions: 
\begin{enumerate}
\item[\foreignlanguage{english}{(a)}] 
$\sum_{i\in I}p_{i}\log\left(p_{i}\log f_{i}\right)>0$. 
\item[\foreignlanguage{english}{(b)}] 
%The polynomial semigroup 
$G=\left\langle f_{1},\dots,f_{s+1}\right\rangle $
is postcritically bounded, i.e. $P(G)\setminus\left\{ \infty\right\} $
is bounded in $\C$. 
%, where $P(G):=\overline{\left\{ \CV(g):g\in G\right\} }$
%and the closure is taken with respect to the spherical metric on %$\Chat$. 
\item[\foreignlanguage{english}{(c)}] \textup{$s=1$. }
\end{enumerate}
Then there exists a Borel dense subset $A$ of $J(G)$ with $\mu _{\vec{p}}(A)=1$ such that for every non-trivial $C\in\T$ and for every 
$z\in A$, we have $\Hol\left(C,z\right)<1.$
%$\mu$-almost everywhere on $J(G)$. 
In particular, every non-trivial $C\in\T$
is non-differentiable $\mu_{\vec{p}}$-almost everywhere on $J(G)$. 
\end{cor}
Note that if we assume that every $f_{i}$ is a polynomial and $P(G)\setminus \{ \infty \} $ 
is bounded in $\C $, then 
$\Lambda (\omega )=0$ for every $\omega \in I^{\Bbb{N}}$, 
thus Corollary~\ref{cor:nondiff} 
%and the preceeding  remark  
implies that 
there exists a Borel dense subset $A$ of $J(G)$ with 
$$\mu _{\vec{p}}(A)=1, \ \dim _{H}(A)\geq 1+\frac{-\sum _{i\in I}p_{i}\log p_{i}}{\sum _{i\in I}
p_{i}\log \deg (f_{i})}>1$$ 
such that for every non-trivial $C\in \T$ and 
for every point $z\in A$, we have 
$$\Hol(C,z)= \frac{-\sum _{i\in I}p_{i}\log p_{i}}{\sum _{i\in I}
p_{i}\log \deg (f_{i})}<1.$$ 
The following is one of the other important applications of 
Corollary~\ref{cor:ergodictheorem}. 
In order to state the result, let $\delta :=\dim _{H}(J(G))$ and 
let $H^{\delta }$ denote the $\delta $-dimensional Hausdorff measure 
on $\Chat .$ Note that by \cite{MR2153926}, we have 
$0<H^{\delta }(J(G))<\infty .$ 
Let $C(J(G))$ be the space of all continuous $\C $-valued 
functions on $\Chat $ endowed with supremum norm. 
Let $L:C(J(G))\rightarrow C(J(G))$ be the operator 
defined by 
$L(\varphi )(z)=\sum _{i\in I}\sum _{f_{i}(y)=z}\phi (y)
\| f_{i}'(y)\| ^{-\delta }$ where $\phi \in C(J(G)), z\in J(G).$ 
By \cite{MR2153926} again,  we have that 
$\gamma =\lim _{n\rightarrow \infty }L^{n}(1)$ $\in C(J(G))$ exists, 
where $1$ denotes the constant function on $J(G)$ 
taking its value $1$, the function 
$\gamma $ is positive on $J(G)$, and 
there exists an $\tilde{f}$-invariant ergodic probability measure 
$\tilde{\nu }$ on $J(\tilde{f})$ such that 
$\pi _{\ast }(\tilde{\nu })=\gamma H^{\delta }/H^{\delta }(J(G))$ 
and supp$\,\pi _{\ast }(\nu )=J(G).$  
 By Corollary~\ref{cor:ergodictheorem} and 
 \cite[Theorem 3.84 (5)]{s11random}, we obtain the following.
\begin{cor} 
\label{cor:Hmeas}
Under the above notations, there exists a Borel dense subset 
$A$ of $J(G)$ with $H^{\delta }(A)=H^{\delta }(J(G))>0$ 
such that for every non-trivial $C\in {\mathcal C}$ and 
for every $z\in A$, we have 
$$ \Hol (C,z)=\frac{-\sum _{i\in I}\log p_{i}
\int _{f_{i}^{-1}(J(G))}\gamma (y) dH^{\delta }(y)}
{\sum _{i\in I}\int _{f_{i}^{-1}(J(G))}\gamma (y)\log \| f_{i}'(y)\| dH^{\delta }(y)}. 
$$  
 \end{cor}

\begin{rem}We remark that a non-trivial $C\in\T$ may possess points of differentiability. In fact, by choosing one of the probability parameters sufficiently small, we can deduce from Corollary  \ref{cor:Hmeas} 
that for every non-trivial $C\in {\mathcal C}$ 
and for $H^{\delta }$-almost every $z\in J(G)$,  we have  $\Hol\left(C,z\right)>1$, $C$ is differentiable at $z$ and 
the derivative of $C$ at $z$ is zero.  
Note that even under the above condition, 
Theorem~\ref{thm:less-than-one} implies that 
there exist an $\alpha <1$ and a dense subset $A$ of 
$J(G)$ with $\dim _{H}(A)>0$ such that 
%if $f_{1},\ldots, f_{s+1}$ are polynomials and 
%if we assume further at least one of the conditions (a)(b)(c) 
%in  Corollary~\ref{cor:nondiff}, then Corollary~\ref{cor:nondiff} 
%implies that for every non-trivial $C\in {\mathcal C}$ and 
%for $\mu_{\vec{p}}$-almost every $z\in J(G)$, 
for every non-trivial $C\in \C $ and for every $z\in A$, 
we have $\Hol (C,z)=\alpha <1$ and $C$ is not differentiable at $z.$ 
In particular, in this case,  we have $\alpha _{-}<1<\alpha _{+}$ 
and we have a different kind of phenomenon 
regarding the (complex) analogues of the Takagi function, 
whereas the original Takagi function does not have this property. 

\end{rem}
We also have the following corollary of Theorem~\ref{main-thm}. 
To state the result, by \cite[Theorem 2.6]{MR1625124}
there exists $k_{0}\in \N $ such that 
for every $k\geq k_{0}$ and for every 
$\omega =(\omega _{i})_{i=1}^{k}\in I^{k}$, 
we have $\min _{z\in f_{\omega }^{-1}(J(G))}\| 
f_{\omega }'(z)\| >1$, where 
$f_{\omega }=f_{\omega _{k}}\circ \cdots \circ f_{\omega _{1}}.$ 
Let $p_{\omega }:=p_{\omega _{k}}\cdots p_{\omega _{1}}$ 
for $\omega =(\omega _{i})_{i=1}^{k}\in I^{k}.$  
\begin{cor}
\label{cor:alpha}
For every $k\geq k_{0}$, 
we have 
$$0<\min _{\omega \in I^{k}}\frac{-\log p_{\omega }}
{\log \max _{z\in f_{\omega }^{-1}(J(G))}\| f_{\omega }'(z)\|}
\leq \alpha _{-}\leq \alpha _{+}\leq 
\max _{\omega \in I^{k}}\frac{-\log p_{\omega }}
{\log \min _{z\in f_{\omega }^{-1}(J(G))}
\| f_{\omega }'(z)\| }<\infty .$$ 
In particular, if 
$p_{i}\min _{z\in f_{i}^{-1}(J(G))}\| f_{i}'(z)\| >1$ for every $i\in I$, then 
 for every non-trivial $C\in \T$ and for every $z\in J(G)$, 
 we have that $\Hol(C,z)\leq \alpha _{+}<1$ and $C$ is not differentiable at $z.$  

\end{cor}
\begin{rem}
\label{rem:norm}
Under assumptions (1)(2)(3), suppose that the maps $f_{i},i\in I, $ are polynomials. 
Then $J(G)\subset \C $. Since the spherical metric and the Euclidian metric are 
equivalent on $J(G)$, it follows that we can replace $\| \cdot \| $ in 
the definition of $\varphi $, Corollaries~\ref{cor:ergodictheorem}, \ref{cor:Hmeas}, 
\ref{cor:alpha} by the modulus $| \cdot |$. 
\end{rem}

\begin{rem}
The function $C_{0}=T_{\vec{p}}$ is continuous (in fact, it is H\"older continuous) on $\Chat$ and 
varies precisely on the Julia set $J(G)$. 
Note that by assumptions (1)(2) and \cite{MR1625124}, we have 
that $J(G)$ is a fractal set with 
$0<\dim _{H}(J(G))<2$.  
The function $C_{0}$ can be interpreted as a complex analogue of the devil's staircase and Lebesgue's singular functions (\cite{s11random}). In fact, 
the devil's staircase is equal to the restriction to $[0,1]$ 
of the function of probability of tending to $+\infty $ when 
we consider random dynamical system on $\Bbb{R}$ such that at every step we choose 
$f_{1}(x)=3x$ with probability $1/2$ and we choose $f_{2}(x)=3x-2$ with probability $1/2.$  
Similarly, Lebesgue's singular function $L_{p}$ 
with respect to the parameter $p\in (0,1), p\neq 1/2$ is equal to the restriction to $[0,1]$ of 
the function of probability of tending to $+\infty $ when we consider random dynamical system 
on $\Bbb{R}$ such that at every step we choose $g_{1}(x)=2x$ with probability $p$ and 
we choose $g_{2}(x)=2x-1$ with probability $1-p.$ 
Note that these are new interpretations of 
the devil's staircase and Lebesgue's singular functions 
obtained in \cite{s11random} by the second author of this paper. 
Similarly, it was pointed out by him that 
the distributional functions of self-similar measures of IFSs of orientation-preserving 
contracting diffeomorphisms $h_{i}$ on $\R $ can be interpreted as the 
functions of probability of tending to $+\infty $ regarding the random dynamical  
systems generated by $(h_{i}^{-1})$ (\cite{s11random}). 
From the above point of view, 
when $G$ is a polynomial semigroup and $L=\{ \infty \} $, 
we call $C_{0}=T_{\vec{p}}$ a devil's coliseum (\cite{s11random}). 
It is well-known (\cite{YHK}) that the function 
$\frac{1}{2}\frac{\partial L_{p}(x)}{\partial p}|_{p=1/2}$ on $[0,1]$ is
equal to the Takagi function 
$\Phi (x)=\sum _{n=0}^{\infty }\frac{1}{2^{n}}\min _{m\in \Bbb{Z}}|2^{n}x-m|$ (also referred to as the Blancmange function), which is a 
famous example of a continuous but nowhere differentiable function on $[0,1]$. 
From this point of view, the first derivatives $C\in {\mathcal C}$ can be interpreted as 
complex analogues of the Takagi function. 
The devil's staircase, Lebesgue's singular functions, the Takagi function and the 
similar functions have been investigated so long in fractal geometry and the related fields. 
In fact, the graphs of these functions have certain kind of self-similarities and 
these functions have many interesting and deep properties.  
There are many interesting studies about 
the original Takagi function and its related topics (\cite{AKsurvey}).   
In \cite{AKTakagi}, many interesting results (e.g. continuity and non-differentiability, $\Hol$der order, the Hausdorff dimension of the graph, 
the set of points where the functions take on their absolute maximum and minimum values)  
of the higher order partial derivatives 
$\frac{\partial ^{n}L_{p}(x)}{\partial p^{n}}|_{p=1/2}$ 
of $L_{p}(x)$ 
with respect to $p$ are obtained. 
The first study of  the complex analogues of the Takagi function was given by the second author 
in \cite{S13Coop}. In particular, some partial results on the pointwise $\Hol$der exponents 
of them were obtained (\cite[Theorem 3.40]{S13Coop}).  
However, it had been an open problem whether the complex analogues of the Takagi function 
vary precisely on the Julia set or not, until this paper was written.   
The results of this paper greatly improve the above results from \cite{S13Coop}. 
In the proofs of the results of this paper,  we use completely new 
ideas and 
systematic approaches which are explained below.    
For the figures of the Julia set $J(G)$ and the graphs 
of $C_0$ and $C_{1}$ which we deal with in this paper when $s=1$, $G$ is a polynomial semigroup and 
$L=\{ \infty \}$, see 
\cite{s11random, S13Coop}. 
\end{rem}
\begin{rem}
The results on the classical Takagi function on  $[0,1]$ give some evidence that the    results  stated in   %Theorem~\ref{thm:multifractal}-(\ref{mf1}) are sharp. 
Theorem~\ref{thm:c0holder} are sharp. 
Indeed, let us consider the function $L_{1/2}$ and 
$\phi _{n}(x)=\frac{\partial ^{n}L_{p}(x)}{\partial p^{n}}|_{p=1/2}$ for $n\ge 1$. Note that 
$\frac{1}{2}\phi _{1}$ is equal to the original Takagi function. 
Since 
we have $L_{1/2}|_{[0,1]}(x)=x$, $L_{1/2}|_{(-\infty ,0)}(x)=0 $ and 
$L_{1/2}|_{(1,\infty )}(x)=1$, the function $L_{1/2}$ is $1$-$\Hol$der (Lipschitz). 
However, in \cite{AKTakagi} it is shown that the functions $\phi _{n}$ on $[0,1]$  are $a$-H\"older for every $a<1$, but not $1$-$\Hol$der continuous. It would be interesting to further investigate this phenomenon for the complex analogues of the Takagi function.  
\end{rem}

\begin{rem}
\label{rem:assump}
We  endow Rat with the topology induced from the 
distance dist$_{\Rat}$ which is defined by 
dist$_{\Rat}(f,g):=\sup _{z\in \Chat}d(f(z),g(z))$. 
%where $d$ denotes the spherical distance on $\Chat .$ 
Then by \cite[Theorem 2.4.1]{sumihyp}, 
the fact $J(G)=\cup _{i\in I}f_{i}^{-1}(J(G))$
(\cite[Lemma 1.1.4]{sumihyp},    
\cite[Remark 3.64]{s11random}, and  
\cite[Theorem 3.24]{S13Coop}), we have that   
 the set 
$$\{ (f_{i})_{i\in I}\in (\mbox{Rat})^{I}:  
\deg (f_{i})\geq 2 \ (i\in I) \mbox{ and 
the conditions (1)(2)(3) hold for } (f_{i})_{i\in I}\} 
$$ is open in $(\mbox{Rat})^{I}.$ 
Also, we have plenty of examples to which we can 
apply the main results of this paper. See Section~\ref{Examples}.

\end{rem}
\begin{rem}
\label{rem:real}
We remark that by using the method in this paper, 
we can show similar results to those of this paper for random dynamical systems of diffeomorphisms  
on $\R $ (or $\R \cup \{ \pm \infty\}$). 
Note that the case of the classical Takagi function $\Phi $ corresponds to the degenerated case 
$\alpha _{-}=\alpha _{+}$ in Theorem~\ref{thm:multifractal}, though in the case of $\Phi $ 
we have the open set condition but do not have the separating condition.  
We emphasize that in this paper we also deal with the non-degenerated case, which seems generic.  
\end{rem}
\vspace{-4mm} 
\begin{rem}
\label{rem:motivation}
We remark that under assumptions (1)(2)(3), 
the iteration of the transition operator $M$ on some $C^{a}(\Chat)$  
%where $Mh:=\sum _{i=1}^{s+1}p_{i}\cdot (h\circ f_{i})$, 
  is 
well-behaved (e.g., there exists an 
$M$-invariant finite-dimensional subspace $U$ of $C^{a}(\Chat )$ such that 
for every $h\in C^{a}(\Chat)$, $M^{n}(h)$ tends to $U$ as $n\rightarrow \infty $ exponentially fast) 
and $M$ has a spectral gap on $C^{a}(\Chat )$ 
(\cite[Lemma 1.1.4(2)]{sumihyp}, \cite[Propositions 3.63, 3.65]{s11random}, \cite[Theorems 3.30, 3.31]{S13Coop}). Note that this is a randomness-induced phenomenon (new phenomenon) in random dynamical systems  
which cannot hold in the deterministic iteration dynamics of rational maps of degree two or more, since 
for every $f\in \mbox{Rat}$ with $\deg (f)\geq 2$, the dynamics of $f$ on $J(f)$ is chaotic.   
Combining the above spectral gap property of $M$ on $C^{a}(\Chat )$ and 
the perturbation theory for linear operators (\cite{Katopert}) implies   
that the  map $\vec{x}=(x_{1},\dots,x_{s})\mapsto T_{(x_{1},\dots,x_{s},1-\sum_{i=1}^{s}x_{i})}\in C^{a}(\Chat)$
is real-analytic in a neighborhood of $\vec{p}$ 
in the space $W:=\{ (q_{i})_{i=1}^{s}\in (0,1)^{s}: \sum _{i=1}^{s}q_{i}<1\}$ 
(\cite[Theorem 3.32]{S13Coop}). 
Thus  
it is very natural and important for the study of 
the random dynamical system to consider the 
higher order partial derivatives of $T_{\vec{p}}$ with respect 
to the probability vectors. Moreover, it is very interesting 
that $C_{\vec{n}}$ is a solution of the functional equation 
$(Id-M)(C_{\vec{n}})=F$, where $F$ is a function associated with lower order 
partial derivatives of $T_{\vec{p}}$ 
(Lemma~\ref{lem:cn-functionalequation}). In fact, by using the spectral gap properties of $M$ on $C^{a}(\Chat)$ 
and the arguments in the proof of \cite[Theorem 3.32]{S13Coop}, 
for any $\vec{n}\in \N_{0}^{s}\setminus \{ 0\}$, 
we can show that (I) 
$C_{\vec{n}}$ is the unique continuous solution of the above 
functional equation 
under the boundary condition  $C_{\vec{n}}|_{S_{G}}=0$ 
and (II) $C_{\vec{n}}=\sum _{j=0}^{\infty }M^{j}(F)$ in $C(\Chat )$ and in 
$C^{\alpha }(\Chat )$ for small $\alpha >0.$ 
Thus, we   have a system of functional equations for elements 
$C_{\vec{n}}$  
(see  Lemma~\ref{lem:cn-functionalequation}). 
Note that this is the first paper to investigate the pointwise $\Hol$der exponents and other properties 
of the  
higher order partial derivatives 
$C_{\vec{n}}$ 
of the functions $T_{\vec{p}}$ of probability of tending to 
minimal sets with respect to the probability parameters regarding random dynamical systems which have several variables of probability parameters.  This is a completely new concept. 
In fact, even in the real line, there has been 
no study regarding the objects similar to the above. 
Even more, in this paper we deal with the 
complex linear combinations of partial derivatives $C_{\vec{n}}$, 
which are of course completely new objects in mathematics 
coming naturally from the study of random dynamical 
systems and fractal geometry.  
We also remark that the original Takagi function is  associated 
with Lebesgue's singular functions, but there has been no study about 
the higher order partial derivatives of the distribution functions of 
singular measures with respect to the parameters.  
\end{rem}
The key in the proof of the main results 
of this paper is to consider the system of  functional equations satisfied by the elements of $\T$ 
(Lemma~\ref{lem:cn-functionalequation}). The composition of these equations along orbits is best described in terms of an associated matrix cocycle $A(\omega ,k)$. By using  combinatorial arguments, 
we show a formula for the components of the matrix $A(\omega ,k)$, and 
we carefully estimate the polynomial growth order of these components, as $k$ tends to infinity 
(Lemma~\ref{lem:matrix-growth}). Combining this 
with some calculations of the determinants of matrices 
which are similar to the Vandermonde determinant 
(Lemma~\ref{detlemma}), 
we deduce the linear independence of the vectors $(C_{\vec{r}}(a)-C_{\vec{r}}(b))_{\vec{r}\leq \vec{n}}$ for certain points 
$a,b\in J(G)$ which are close to a given point $x_{0}\in J(G)$ 
(Proposition~\ref{prop:fullrank}). Here, 
$\vec{r}\leq \vec{n}$ means that $r_{i}\leq n_{i}$ for each $i.$ 
 From the linear independence of these vectors 
 we deduce that a certain linear combination of  vectors $(C_{\vec{r}}(a)-C_{\vec{r}}(b))_{\vec{r}\leq \vec{n}}$ 
 is bounded away from zero (Lemma~\ref{lem:eta-keylemma}). This gives us the upper bound of 
 the pointwise $\Hol$der exponents of $C\in \T.$  
 Note that this argument is the key to prove 
 Theorem~\ref{main-thm} and it is 
 the crucial point to derive that the elements $C\in \T$ are not locally constant in any point of the  Julia set (Corollary~\ref{cor:varies}). 
   We emphasize that those ideas are  very new and 
 they give us strong and systematic tools  to analyze random dynamical systems, 
 singular functions, fractal functions and  
 other related topics.  

In Section~\ref{Examples}, 
we give plenty of examples which illustrate the main results of this paper. 
In Section~\ref{Preliminaries} we give some fundamental tools of rational semigroups 
and random complex dynamics. 
In Section~\ref{Functional} we describe the system of functional equations for the elements of 
${\mathcal C}$ and we estimate the growth order of components of associated matrix cocycles. 
In Section~\ref{Proof}, we give the proof of Theorem~\ref{main-thm}, by using the results 
from Section~\ref{Functional}.  
In Section~\ref{section:pfmulti},  
we present the detailed version Theorem~\ref{thm:mf-for-hoelderexponent} 
 of Theorem~\ref{thm:multifractal} and we give the proof of it  
 by using Theorem~\ref{main-thm} and 
some results from \cite[Theorem 1.2]{JS13b}. 
Also, we  give the proof of Theorem~\ref{thm:c0holder}
by using the argument in the proof of Theorem~\ref{main-thm} and by 
developing some ideas from \cite{kessenondifferentiabilityminkowski, MR2576266}. 
In Section~\ref{section:pfl-t-o}, we give the proof of 
Theorem~\ref{thm:less-than-one} by combining that $C_{0}$ is $\alpha _{-}$-$\Hol$der continuous on $\Chat $
 (Theorem~\ref{thm:c0holder}),  
 the multifractal analysis on the pointwise $\Hol$der exponents of $C_{0}$ 
(Theorems~\ref{thm:multifractal} and \ref{thm:mf-for-hoelderexponent}), 
an  argument on the Lipschitz functions on $\C $  
and the result $0<\dim _{H}(J(G))<2$, which follows from the assumptions (1) and  (2) (\cite{MR1625124}).  
\selectlanguage{english}%
\section{Examples}
\label{Examples}
In this section, we give some examples which illustrate 
the main results of this paper. 

For $f\in $Rat, we set $F(f):=F(\langle f\rangle ), J(f):=J(\langle f\rangle )$, 
and $P(f)=P(\langle f\rangle ).$  
We denote by ${\mathcal P}$ the set of polynomials of degree two or  more. 
For $g\in {\mathcal P}$, we denote by $K(g)$ the filled-in Julia set. 
If $G$ is a rational semigroup and if $K$ is a non-empty 
compact subset of $\Chat $ such that $g(K)\subset K$ for each 
$g\in G$, then Zorn's lemma implies that there exists a minimal 
set $L$ of $G$ with $L\subset K$ (\cite[Remark 3.9]{s11random}).  
%From this result, we obtain the following. 

The following propositions show us  several methods to produce many examples of 
$(f_{1},\ldots, f_{s+1})\in (\mbox{Rat})^{s+1}$ which satisfy assumptions (1)(2)(3) of this paper. 
For such elements $(f_{1},\ldots, f_{s+1}) $ and for every $\vec{p}=(p_{i})_{i=1}^{s}\in (0,1)^{s}$
 with $\sum _{i=1}^{s}p_{i}<1$, we can apply the results Theorems~\ref{main-thm}, 
\ref{thm:multifractal}, \ref{thm:c0holder},  
\ref{thm:less-than-one}, \ref{thm:alphachaos} and   
Corollaries~\ref{cor:varies}, \ref{cor:ergodictheorem}, \ref{cor:Hmeas} and \ref{cor:alpha} in Section~\ref{Introduction}. 
\begin{prop}
\label{prop:hypsub}
Let $(g_{1},\ldots ,g_{s+1})\in (\mbox{{\em Rat}})^{s+1}$ with 
$\deg (g_{i})\geq 2, i=1\ldots, s+1.$ 
Suppose that $\langle g_{1},\ldots, g_{s+1}\rangle $ is hyperbolic,  
%the Julia sets $J(g_{i}), i=1,\ldots, s+1$ are mutually disjoint, 
$J(g_{i})\cap J(g_{j})=\emptyset $ for every 
$(i,j)$ with $i\neq j$,  
and that
there exist at least two distinct minimal sets of $\langle g_{1},\ldots, g_{s+1}\rangle .$ 
Then there exists $m\in \Bbb{N}$ such that for every $n\in \Bbb{N}$ with 
$n\geq m$, setting $f_{i}=g_{i}^{n}, i=1,\ldots, s+1$, 
the element $(f_{1},\ldots, f_{s+1})$ satisfies assumptions (1)(2)(3) of this paper. 
\end{prop}
\begin{proof}
Let $H=\langle g_{1},\ldots, g_{s+1}\rangle .$ 
Since $J(g_{i}), i=1,\ldots, s+1$ are mutually disjoint and since 
attracting cycles of $g_{i}$ are included in $F(H)=\Chat \setminus J(H)$, 
there exists $m\in \Bbb{N} $ such that  for every $n\geq m$, 
setting  $f_{i}=g_{i}^{n}, i=1,\ldots,$ the sets 
$f_{i}^{-1}(J(H)), i=1,\ldots, s+1$, are mutually disjoint. 
Let $G=\langle f_{1},\ldots, f_{s+1}\rangle .$ 
Then $G$ is a subsemigroup of $H.$ Thus 
$F(H)\subset F(G)$ and $P(G)\subset P(H).$ 
Hence $P(G)\subset P(H)\subset F(H)\subset F(G).$ Therefore 
$G$ is hyperbolic. Moreover, since $J(G)\subset J(H)$, 
the sets   $f_{i}^{-1}(J(G)), i=1,\ldots, s+1$, are mutually disjoint. 
Let $L_{1}$ and $L_{2}$ be two distinct minimal sets of $H.$ 
Then for every $g\in H$ and for every $i=1,2$, we have $g(L_{i})\subset L_{i}.$ 
In particular, for every $f\in G$ and for every $i=1,2,$ $f(L_{i})\subset L_{i}.$ 
By \cite[Remark 3.9]{s11random}
%the remark before Proposition~\ref{prop:s1ex}, 
it follows that for every $i=1,2,$ there exists a minimal set
$L_{i}'$ of $G$ with $L_{i}'\subset L_{i}.$ 
Hence $(f_{1},\ldots, f_{s+1})$ satisfies assumptions (1)(2)(3) of this paper. 
\end{proof}

\begin{prop}
\label{p:ratex}
Let $(g_{1},\ldots, g_{s+1})\in (\mbox{{\em Rat}})^{s+1}$ with $\deg (g_{i})\geq 2, i=1,\ldots ,s+1.$ 
Suppose that $\cup _{i=1}^{s+1}P(g_{i})\subset \cap _{i=1}^{s+1}
F(g_{i})$, that $J(g_{i})\cap J(g_{j})=\emptyset $ for every 
$(i,j)$ with $i\neq j$, and that there exist two non-empty compact subsets
$K_{1}, K_{2}$ of $\Chat $ with $K_{1}\cap K_{2}=\emptyset $ 
such that $g_{i}(K_{j})\subset K_{j}$ for every $i=1,\ldots, s+1$ and 
for $j=1,2.$ Then there exists $m\in \Bbb{N}$ such that for every $n\in \Bbb{N}$ with 
$n\geq m$, setting $f_{i}=g_{i}^{n}, i=1,\ldots, s+1$, 
the element $(f_{1},\ldots, f_{s+1})$ satisfies assumptions (1)(2)(3) of this paper. 
\end{prop}
\begin{proof}
Let $\epsilon >0$ be so small that 
$B(\cup _{i=1}^{s+1}P(g_{i}), 2\epsilon )\subset \cap _{i=1}^{s+1}F(g_{i})$ 
and $B(J(g_{i}),2\epsilon )\cap B(J(g_{j}),2\epsilon )=\emptyset $ 
for every $(i,j)$ with $i\neq j.$ 
Let $m\in \Bbb{N}$ be a sufficiently large number. 
Let $n\geq m$ and let $f_{i}=g_{i}^{n}.$ 
Let $G=\langle f_{1},\ldots, f_{s+1}\rangle .$  
Let $A_{k}:=B(\cup _{i=1}^{s+1}P(g_{i}), k\epsilon )$ for each $k=1,2.$ 
Then taking $m$ so large, we have 
$f_{i}(A_{2})\subset A_{1}$ for every $i=1,\ldots, s+1.$ 
It implies $P(G)\subset \overline{A_{1}}\subset A_{2}\subset F(G).$ Hence $G$ is hyperbolic. 
Moreover, by \cite[Remark 3.9]{s11random}, there exists a minimal set $L_{j}$ of $G$ 
with $L_{j}\subset K_{j}$, for every $j=1,2.$ 
By Proposition~\ref{prop:hypsub}, the statement of our proposition holds. 
\end{proof}

Combining \cite[Remark 3.9]{s11random} with \cite[Proposition 6.1]{s11random}, we 
also obtain the following.  

\begin{prop}
\label{prop:s1ex}
Let $f_{1}\in {\mathcal P}$  
be hyperbolic,  i.e., $P(f_{1})
\subset F(f_{1}).$ 
Suppose that $\mbox{{\em Int}}(K(f_{1}))\neq \emptyset $, where 
{\em Int} denotes the set of interior points. 
%and 
%$K(f_{1})$ denotes the fillied-in Julia set of $f_{1}.$ 
Let 
$b\in \mbox{{\em Int}}(K(f_{1}))$ be a point.  
Let $d\in \Bbb{N}$ with $d\geq 2.$ 
Suppose that $(\deg (f_{1}),d)\neq (2,2)$. 
Then there exists a number $c>0$ such that 
for each $\lambda \in \{ \lambda \in \C: 0<|\lambda |<c\} $, 
setting $f_{2,\lambda }(z):=\lambda (z-b)^{d}+b$, 
we have the following. 
\begin{itemize}
\item[{\em (1)}]
$(f_{1},f_{2,\lambda })$ satisfies assumptions {\em (1)(2)(3)}  
of this paper with $s=1.$   

\item[{\em (2)}] 
If $J(f_{1})$ is connected, then  
$P(\langle f_{1},f_{2,\lambda }\rangle )\setminus \{ \infty\} $ 
is bounded in $\C .$

\end{itemize}
  
\end{prop}
Thus combining the above with Remark~\ref{rem:assump}, 
we obtain that for any $(f_{1},f_{2,\lambda })$ in the above, 
there exists a neighborhood $V$ of $(f_{1},f_{2,\lambda })$ 
in $(\mbox{Rat})^{2}$ such that for  every $(g_{1},g_{2})\in V$, 
assumptions (1)(2)(3) of this paper are satisfied and  
Theorems~\ref{main-thm}, \ref{thm:multifractal}, \ref{thm:c0holder}, 
\ref{thm:less-than-one},  \ref{thm:alphachaos} and   
Corollaries~\ref{cor:varies}, \ref{cor:ergodictheorem},\ref{cor:Hmeas} and \ref{cor:alpha} in  Section~\ref{Introduction} hold. 
Also, 
%denoting by ${\mathcal P}$ the set of polynomials of degree two or more, 
endowing ${\mathcal P}$ with the relative topology from Rat, we have that 
there exists an open neighborhood $W$ of $(f_{1},f_{2,\lambda })$ in ${\mathcal P}^{2}$ 
such that for every $(g_{1},g_{2})\in W$ and for every $\vec{p}=p_{1}\in (0,1)$, Corollary~\ref{cor:nondiff} holds.   
%Note that if $J(f_{1})$ is connected and the set $V$ above is small 
%enough, then 
%$P(\langle g_{1},\rangle g_{2}\rangle \setminus \{ \infty \} $ 
%is bounded in $\Chat $ 
%and  Corollary~\ref{cor:nondiff} implies that 
%there exists a Borel subset $A$ of $J(\langle g_{1},g_{2}\rangle )$ 
%with $\mu _{\vec{p},g_{1},g_{2}}(A)=1$, where 
%$\mu _{\vec{p},g_{1},g_{2}}$ is the  measure associated with 
%$(g_{1},g_{2})$, such that for each non 

\begin{example}
Let $(f_{1},f_{2})\in {\mathcal P}^{2}$ be an element such that 
$\langle f_{1},f_{2}\rangle $ is hyperbolic, 
$P(\langle f_{1},f_{2}\rangle )\setminus \{ \infty \} $ is bounded in $\C $ and 
$J(\langle f_{1}, f_{2}\rangle )$ is disconnected. Note that 
there are plenty of examples of such elements $(f_{1},f_{2})$ (Proposition~\ref{prop:s1ex}, 
\cite{s11pb1, twogenerator}). 
Then by  \cite[Theorems 1.5, 1.7]{s09cohom}, we have that 
$f_{1}^{-1}(J(G))\cap f_{2}^{-1}(J(G))=\emptyset $ where $G=\langle f_{1},f_{2}\rangle. $ 
Thus $(f_{1},f_{2})$ satisfies assumptions (1)(2)(3) of this paper with $s=1$ and 
all results in Section~\ref{Introduction} hold for $(f_{1},f_{2})$ and for every 
$\vec{p}=p_{1}\in (0,1).$    
 \end{example}  
\begin{example} 
%(\cite[Example 6.2]{s11random},\cite[Example 6.2]{S13Coop}.) 
Let $g_{1}(z)=z^{2}-1, g_{2}(z)=z^{2}/4,$ and  let $f_{i}=g_{i}\circ g_{i}$, $i=1,2.$ 
Let $\vec{p}=p_{1}=1/2.$ Let $G=\langle f_{1},f_{2}\rangle. $ 
Then $(f_{1},f_{2})$ satisfies the assumptions (1)(2)(3) of this paper with $s=1$ and 
$P(G)\setminus \{ \infty \}$ is bounded in $\C $ (\cite[Example 6.2]{s11random},\cite[Example 6.2]{S13Coop}). 
Thus for this $(f_{1},f_{2})$, all results of Section~\ref{Introduction} hold. 
In particular, every non-trivial $C\in \T$ is $\Hol$der continuous on $\Chat$ and 
varies precisely on the Julia set $J(G)$ (Corollary~\ref{cor:varies}).  Moreover, 
by Corollary~\ref{cor:nondiff}, 
% and Theorem~\ref{thm:multifractal}, 
there exists a Borel dense subset $A$ of $J(G)$ with $\mu _{\vec{p}}(A)=1, \ \dim _{H}(A)\geq \dim _{H}(\mu _{\vec{p}})=\frac{3}{2}$  
such that for every non-trivial $C\in \T$ and for every $z\in A$, we have
 $\alpha _{-}\leq \Hol(C,z)=\frac{1}{2}\leq \alpha _{+}$ and $C$ is not differentiable at $z.$   
% Also, we have  (see the remark before Corollary~\ref{cor:nondiff}). 
 %In fact, 
 %by using \cite[]{MR1767945}
  For the figures of $J(G)$ and the graphs of $C_{0}, C_{1}$ with $L=\{ \infty \}$, 
 see \cite[Figures 2,3,4]{S13Coop}.  
Note that Theorem~\ref{thm:multifractal} implies that $\alpha _{-}<\alpha _{+}$ 
for every probability vector (parameter) $\vec{p}' \in (0,1).$   
  \end{example} 
\begin{example} 
%(\cite[Example 5.4]{rcddc}.) 
Let $\lambda \in \C $ with $0<|\lambda |\leq 0.01$ and let 
$f_{1}(z)=z^{2}-1, f_{2}(z)=\lambda z^{3}.$ 
Then by \cite[Example 5.4]{rcddc}, the element $(f_{1},f_{2})$ satisfies 
assumptions (1)(2)(3) of this paper with $s=1$ and 
$P(\langle f_{1},f_{2}\rangle )\setminus \{ \infty \} $ is bounded in $\C .$ Thus all results in Section~\ref{Introduction} 
hold for $(f_{1},f_{2})$ and for every probability vector (parameter) $\vec{p}=p_{1}\in (0,1).$ 
Thus, setting $p_{1}=\frac{1}{2}$, $G=\langle f_{1},f_{2}\rangle $ and $L=\{ \infty \}$, 
every non-trivial $C\in \T $ is $\Hol$der continuous on $\Chat $ and 
varies precisely on $J(G)$, and 
Corollary~\ref{cor:nondiff} 
%and Theorem~\ref{thm:multifractal} 
implies that 
there exists a Borel dense subset $A$ of $J(G)$ with 
$\mu _{\vec{p}}(A)=1$ and $\dim _{H}(A)\geq 1+\frac{2\log 2}{\log 2+\log 3}\fallingdotseq 1.7737 
$ such that for every non-trivial $C\in \T $ and for every 
$z\in A$, we have 
$\alpha _{-}\leq \Hol(C,z)=\frac{2\log 2}{\log 2+\log 3} (\fallingdotseq 0.7737) \leq \alpha _{+}$ 
and $C$ is not differentiable at $z.$ 
Also, by Theorem~\ref{thm:multifractal}, we have 
$\alpha _{-}<\alpha _{+}$ for every $\vec{p}'\in (0,1).$ 
 \end{example}   
\begin{example}
% (\cite[Proposition 6.3]{s11random}.) 
Let $g_{1}, g_{2}\in {\mathcal P}$ be hyperbolic. 
Suppose that $(J(g_{1})\cup J(g_{2}))\cap (P(g_{1})\cup P(g_{2}))=\emptyset $, 
$K(g_{1})\subset \mbox{Int}(K(g_{2}))$, and the union of attracting cycles of $g_{2}$ in $\C $ is included in 
Int$(K(g_{1})).$ Then by \cite[Proposition 6.3]{s11random}, there exists an $m\in \Bbb{N}$ 
such that for each $n\in \Bbb{N}$ with $n\geq m$, setting $f_{1}=g_{1}^{n}, f_{2}=g_{2}^{n}$, 
we have that $(f_{1},f_{2})$ satisfies assumptions (1)(2)(3) of this paper with $s=1.$ 
Thus all statements of the results in Section~\ref{Introduction} hold for $(f_{1},f_{2})$ and for every 
$\vec{p}=p_{1}\in (0,1).$ 
 
 \end{example} 
The following proposition provides us a method to 
construct examples of $(f_{1},\ldots, f_{s+1})\in {\mathcal P}^{s+1}$ 
for which  (1)(2)(3) hold and $P(\langle f_{1},\ldots, f_{s+1}\rangle )\setminus 
\{ \infty \}   $ is bounded in $\C .$ 
For such elements $(f_{1},\ldots, f_{s+1})$ and for every 
$\vec{p}\in (0,1)^{s}$ with $\sum _{i=1}^{s}p_{i}<1$, we can apply 
all the results in Section~\ref{Introduction}. 
\begin{prop}
\label{prop:pbbsany}
Let $g_{1},\ldots, g_{s+1}\in {\mathcal P}$ be hyperbolic and suppose that 
$J(f_{i})$ is connected for every $i=1,\ldots, s+1.$ 
Suppose that $J(f_{i})\subset \mbox{Int}(K(f_{i+1}))$ for every $i=1,\ldots, s$. 
Suppose also that $\cup _{i=2}^{s+1}P(g_{i})\setminus \{ \infty \} \subset \mbox{Int}(K(f_{1})).$ 
Then there exists an $m\in \N $ such that for every $n\in \N $ with $n\geq m$, 
setting $f_{i}=g_{i}^{n}, i=1,\ldots, s+1$, 
the element $(f_{1},\ldots, f_{s+1})$ satisfies assumptions (1)(2)(3) and 
$P(\langle f_{1}, \ldots ,f_{s+1}\rangle )\setminus \{ \infty \} $ is bounded in $\C .$   
\end{prop}
\begin{proof}
Let $n\in \N $ be large enough and let $f_{i}=g_{i}^{n}.$ 
Then there exists a compact subset $A$ of $\mbox{Int}(K(f_{1}))$ such that 
$\cup _{i=2}^{s+1}f_{i}(K(f_{1}))\subset A.$ Also, 
$\overline{\cup _{r=1}^{\infty }f_{1}^{r}(A\cup P(f_{1})\setminus \{ \infty \} )}
\subset \mbox{Int}(K(f_{1})).$ Hence $P(G)\setminus \{ \infty \} \subset 
\mbox{Int}(K(f_{1}))$, where 
$G=\langle f_{1},\ldots, f_{s+1}\rangle .$ 
In particular, $P(G)\setminus \{ \infty \} $ is bounded in $\C .$ 
Since $\cup _{i=1}^{s+1}f_{i}(K(f_{1}))\subset K(f_{1})$, 
we obtain that Int$(K(f_{1}))\subset F(G).$ 
Hence,  $P(G)\subset F(G)$ and $G$ is hyperbolic. 
Let $B=K(g_{s+1})\setminus \mbox{Int}(K(f_{1})).$ 
By taking $n$ large enough, we may assume that 
$\cup _{i=1}^{s+1}f_{i}^{-1}(B)\subset B$ and 
the sets $f_{i}^{-1}(B), i=1,\ldots, s+1$, are mutually disjoint. 
Since $\cup _{i=1}^{s+1}f_{i}^{-1}(B)\subset B$, \cite[Corollary 3.2]{MR1397693}
   implies that $J(G)\subset B.$ Hence,  
   the sets $f_{i}^{-1}(J(G)), i=1,\ldots, s+1$, are mutually disjoint. 
  Since $\cup _{i=1}^{s+1}f_{i}(K(f_{1}))\subset K(f_{1})$, 
  %the remark before Proposition~\ref{prop:s1ex} 
  \cite[Remark 3.9]{s11random} 
  implies that there exists a minimal set $L$ of $G$ with $L\subset K(f_{1})$. 
  Thus,  there exist at least two minimal sets of $G$. Hence,  
  $(f_{1},\ldots, f_{s+1})$ satisfies assumptions (1)(2)(3) of our paper and 
  $P(G)\setminus \{ \infty \} $ is bounded in $\C .$ 
\end{proof}
\begin{example}
Let $g_{1}(z)=z^{2}-1$ and let $g_{i}(z)=\frac{1}{10i}z^{2}, i=2,\ldots, s+1.$ 
Then $(g_{1},\ldots, g_{s+1})$ satisfies the assumptions of Proposition~\ref{prop:pbbsany}. 
Note that $z^{2}-1$ can be replaced by any hyperbolic element $f\in {\mathcal P}$ 
with connected Julia set such that $J(f)\subset \{ z\in \C : |z|<10\}$ and 
$0\in \mbox{Int}(K(f)).$  
\end{example}

From one element $(g_{1},\ldots, g_{m})\in (\mbox{Rat})^{m}$ which 
satisfies assumptions (1)(2)(3) (with $s+1=m$), we obtain many elements 
which satisfy  assumptions (1)(2)(3) of our paper as follows. 
\begin{prop}
\label{prop:sg2}
Let $(g_{1},\ldots ,g_{m})\in (\mbox{{\em Rat}})^{m}$ with $\deg (g_{i})\geq 2, i=1,\ldots, m$, 
and suppose that $(g_{1},\ldots ,g_{m})$ satisfies assumptions (1)(2)(3) of this paper.  
Let $n\in \Bbb{N}$ with $n\geq 2$ and let $f_{1},\ldots, f_{s+1}$ be 
 mutually distinct elements of $\{ g_{\omega _{n}}\circ \cdots \circ g_{\omega _{1}}\mid 
(\omega _{1},\ldots, \omega _{n})\in \{ 1,\ldots, m\} ^{n}\} $ where $s\geq 1.$  
Then we have the following. 
\begin{itemize}
\item[{\em (I)}] 
$(f_{1},\ldots, f_{s+1})$ satisfies assumptions (1)(2)(3) of this paper. Thus  
all statements in Theorems~\ref{main-thm}, \ref{thm:multifractal}, 
\ref{thm:c0holder}, \ref{thm:less-than-one}, \ref{thm:alphachaos}  and   
Corollaries~\ref{cor:varies}, \ref{cor:ergodictheorem}, \ref{cor:Hmeas} and \ref{cor:alpha} in Section~\ref{Introduction} hold for $(f_{1},\ldots, f_{s+1})$, for every minimal set $L$  of  
$\langle f_{1},\ldots, f_{s+1}\rangle $ 
and 
for every $\vec{p}=(p_{1},\ldots ,p_{s})\in (0,1)^{s}$ with 
$\sum _{i=1}^{s}p_{i}<1.$ 
\item[{\em (II)}] 
If, in addition to the assumption, 
$(f_{1},\ldots, f_{s+1})$ $\in {\mathcal P}^{s+1}$, then statement (1) in 
Corollary~\ref{cor:nondiff} holds for 
$(f_{1},\ldots, f_{s+1})$ and for every $\vec{p}$,  and 
statement (2) in Corollary~\ref{cor:nondiff}  holds for 
$(f_{1},\ldots, f_{s+1})$ and for every $\vec{p}$ provided that 
one of  (a)(b)(c) in the assumption of Corollary~\ref{cor:nondiff} (2) holds.  

\item[{\em (III)}]
If, in addition to the assumption of our proposition, 
$(g_{1},\ldots, g_{m})\in {\mathcal P}^{m}$ and 
$P(\langle g_{1},\ldots, g_{m}\rangle )\setminus \{ \infty \} $ 
is bounded in $\C $, then 
$P(\langle f_{1},\ldots, f_{s+1}\rangle )\setminus \{ \infty \} $ is 
bounded in $\C  .$ Thus,  statement (2) in Corollary~\ref{cor:nondiff} holds 
for $(f_{1},\ldots ,f_{s+1})$ and for every $\vec{p}.$ 
\end{itemize} 
\end{prop}
\begin{proof}
Let $H=\langle g_{1},\ldots, g_{m}\rangle $ and let 
$G=\langle f_{1},\ldots, f_{s+1}\rangle .$ Then 
$G$ is a subsemigroup of $H.$ Hence,  
$F(H)\subset F(G)$ and $P(G)\subset P(H).$ 
Since $H$ is hyperbolic, we have 
$P(G)\subset P(H)\subset F(H)\subset F(G).$ Thus,  
$G$ is hyperbolic.  Hence,  $(f_{1},\ldots, f_{s+1})$ satisfies assumption (1) of our paper. 
Since the sets $ g_{i}^{-1}(J(H)): i=1,\ldots, m, $ are mutually disjoint, 
we have that the sets 
$ (g_{\omega _{n}}\circ \cdots \circ g_{\omega _{1}})^{-1}(J(H)),  
(\omega _{1},\ldots, \omega _{n})\in \{ 1,\ldots, m\} ^{n}, $ are mutually disjoint. 
Since $J(G)\subset J(H)$, it follows that 
the sets $f_{i}^{-1}(J(G)), i=1,\ldots, s+1, $ are mutually disjoint. 
Hence $(f_{1},\ldots, f_{s+1})$ satisfies assumption (2) of our paper. 
Since $(g_{1},\ldots, g_{m})$ satisfies assumption (3) of our paper, 
there exist at least two distinct minimal sets $L_{1}$ and $L_{2}$ of 
$\langle g_{1},\ldots, g_{m}\rangle. $ 
Therefore for every $g\in \langle g_{1},\ldots,g_{m}\rangle $ and for every $i=1,2$, 
we have $g(L_{i})\subset L_{i}.$ In particular, for  every 
$f\in \langle f_{1},\ldots ,f_{s+1}\rangle $, 
$f(L_{i})\subset L_{i}.$ 
By 
\cite[Remark 3.9]{s11random} 
%the remark before Proposition~\ref{prop:s1ex}, 
it follows that for every $i=1,2,$ there exists a minimal set
$L_{i}'$ of $\langle f_{1},\ldots ,f_{s+1}\rangle $ with $L_{i}'\subset L_{i}.$ 
Hence,  $(f_{1},\ldots, f_{s+1})$ satisfies assumption (3) of our paper. 
If, in addition to the assumption of our proposition, 
$(g_{1},\ldots, g_{m})\in {\mathcal P}^{m}$ and 
$P(H)\setminus \{ \infty \} $ 
is bounded in $\C $, then since $P(G)\setminus \{ \infty \} 
\subset P(H)\setminus \{ \infty \} $, we obtain that 
$P(G)\setminus \{ \infty  \} $ is 
bounded in $\C  .$ 
\end{proof} 
Regarding Remark~\ref{rem:assump}, we also have the following. 
\begin{lem}
\label{lem:pbbopen}
Let $s\geq 1$ and let $I=\{ 1,\ldots, s+1\}.$ Then the set 
$$\{ (f_{i})_{i\in I}\in {\mathcal P}^{I}: (f_{i})_{i\in I} \mbox{ satisfies assumptions (1)(2)(3) and }
P(\langle f_{1},\ldots ,f_{s+1}\rangle )\setminus \{ \infty \} \mbox{ is bounded in }\C \}$$ 
is open in ${\mathcal P}^{I}.$ 

\end{lem}  

\begin{proof}
By \cite[Lemma 5.4]{s11pb3}, we have that 
the set of elements $(f_{i})_{i\in I}\in {\mathcal P}^{I}$ for which  
assumption (1) holds and $P(\langle f_{1},\ldots ,f_{s+1}\rangle )\setminus 
\{ \infty \} $ is bounded is open in 
${\mathcal P}^{I}.$ 
Combining this with Remark~\ref{rem:assump}, 
we see that the statement of our lemma holds.  
\end{proof} 
We remark that the above examples, propositions and lemma in this section and Remark~\ref{rem:assump} 
imply that we have plenty of examples to which we can apply the results in Section~\ref{Introduction}. 

We give  examples to which we can apply Corollary~\ref{cor:alpha}. 
\begin{lem}
\label{lem:ndevery}
Let $(g_{1},\ldots, g_{s+1})$ be an element which satisfies assumptions (1)(2)(3). 
Let $\vec{p}=(p_{i})_{i=1}^{s}\in (0,1)^{s}$ with $\sum _{i=1}^{s+1}p_{i}<1.$ 
Let $p_{s+1}=1-\sum _{i=1}^{s}p_{i}.$ 
Then there exists an $m\in \N $ such that for every $n\in \N $ with $n\geq m$, 
setting $f_{i}=g_{i}^{n}, i=1\ldots, s+1$, and setting $G:=\langle f_{1},\ldots, f_{s+1}\rangle $,  we have that 
$(f_{1},\ldots, f_{s+1})$ satisfies assumptions (1)(2)(3) and  
$p_{i}\min _{z\in f_{i}^{-1}(J(G))}\| f_{i}'(z)\| >1$ for every $i=1,\ldots ,s+1$. 
Thus,  for every  minimal set $L$ of 
$\langle f_{1},\ldots, f_{s+1}\rangle $, and for every $z\in J(G)$, we have that 
every non-trivial $C\in \T $ satisfies $\Hol(C,z)\leq \alpha _{+}<1$ and 
$C$ is not differentiable at $z.$   
\end{lem}
\begin{proof}
By Proposition~\ref{prop:sg2}, there exists an $m\in \N $ such that 
for every $n\in \N $ with $n\geq m$, 
setting $f_{i}=g_{i}^{n}, i=1\ldots, s+1$, 
%and setting $G:=\langle f_{1},\ldots, f_{s+1}\rangle $,  we have that 
we have that $(f_{1},\ldots, f_{s+1})$ satisfies assumptions (1)(2)(3). 
Since $H:=\langle g_{1},\ldots ,g_{s+1}\rangle $ is hyperbolic, the expanding property of 
$H$ 
on $J(H)$ (\cite[Theorem 2.6]{MR1625124}) implies that 
if $n$ is large enough, then 
$p_{i}\min _{z\in f_{i}^{-1}(J(G))}\| f_{i}'(z)\| >1$ for every $i=1,\ldots ,s+1$, 
where $G=\langle f_{1},\ldots ,f_{s+1}\rangle .$  
Combining this with Corollary~\ref{cor:alpha}, we obtain that,  
  for  each  minimal set $L$ of 
$G$, and for every $z\in J(G)$, we have that 
every non-trivial $C\in \T $ satisfies $\Hol(C,z)\leq \alpha _{+}<1$ and 
$C$ is not differentiable at $z.$   
%Let $G=\langle f_{1},\ldots, f_{s+1}\rangle .$ 
\end{proof}
\section{Preliminaries}
\label{Preliminaries} 

In this section, we recall some fundamental facts on rational semigroups and 
random complex dynamics which are needed in the proofs of the main results of this paper. 

Let $G$ be a rational semigroup and let $z\in \Chat $. 
The backward orbit $G^{-}(z)$ of $z$ and the set of exceptional points 
$E(G)$ are defined by $G^{-}(z):=
\cup _{g\in G}g^{-1}(z)$ and 
$E(G):=\{ z\in \Chat :\mbox{card}(G^{-}(z))<\infty \} .$ 
We say that a set $A\subset \Chat $ is $G$-backward invariant, 
if $g^{-1}(A)\subset A$ for each $g\in G$, and we say that 
$A$ is $G$-forward invariant, if $g(A)\subset A$, for each $g\in G.$ 

The following was proved in \cite{MR1397693} (see also  \cite[Lemma 2.3]{MR1767945}, 
\cite{MR2900562}). 

\begin{lem}Let $G$ be a rational semigroup which has an element of degree two or more. Then we have the following.
\begin{itemize}
\item[{\em (a)}] 
$F(G)$ is $G$-forward invariant and $J(G)$ is $G$-backward invariant. 

\item[{\em (b)}]
$J(G)$ is a perfect set, 
\item[{\em (c)}] 
card$(E(G))\leq 2$.
\item[{\em (d)}] 
If $z\in \Chat \setminus E(G)$, then 
$J(G)\subset \overline{G^{-}(z)}. $ 
In particular, if $z\in J(G)\setminus E(G)$, then 
$\overline{G^{-}(z)}=J(G).$  
\item[{\em (e)}] 
$J(G)$ is the smallest closed subset of $\Chat$ 
containing at least three points which is $G$-backward invariant. 
\item[{\em (f)}]

$J(G)=\overline{\{ z\in \Chat: z\mbox{ is a repelling fixed point of some }g\in G\} }
=\overline{\cup _{g\in G}J(g)}.$ 

\end{itemize}

\end{lem}
 The following lemma (\cite[Lemma 1.1.4]{sumihyp}) is easy to see but important. 
 \begin{lem}
 Let $G$ be a rational semigroup generated by 
 $\{ f_{1},\ldots, f_{m}\}.$ Then 
 $J(G)=\cup _{j=1}^{m}f_{j}^{-1}(J(G)).$ 
  \end{lem} 
  
We remark that by \cite{MR1625124} and \cite[Remark 5]{MR2153926},  assumption (1) of this paper 
is equivalent to 
the property that the associated skew product map is expanding in the sense of \cite{MR2153926} and  \cite{JS13b}.  Combining assumptions (1)(2) of our paper and \cite[Theorem 2.14 (2), Lemma 2.4]{MR1827119}, 
we obtain the following.  
\begin{lem}
\label{lem:exp}
Suppose that $(f_{1},\ldots. f_{s+1})$ satisfies assumptions (1)(2) of our paper. 
Let $G=\langle f_{1},\ldots, f_{s+1}\rangle $, 
let $I=\{ 1,\ldots, s+1\} $ and let $\tilde{f}$ be the skew product map 
associated with $(f_{1},\ldots, f_{s+1}).$ 
Then $J(\tilde{f})=\cup _{\omega \in I^{\Bbb{N}}}(\{ \omega \} \times J_{\omega })$ 
and $J(G)=\bigsqcup _{\omega \in I^{\Bbb{N}}}J_{\omega }$, 
where $\bigsqcup $ denotes the disjoint union. 
Also, for every $\omega =(\omega _{i})_{i\in \Bbb{N}}\in I^{\Bbb{N}}$, we have 
$f_{\omega _{1}}(J_{\omega })=J_{\sigma (\omega )}$ and 
$f_{\omega _{1}}^{-1}(J_{\sigma (\omega )})=J_{\omega }.$ 
\end{lem}
We remark that $\pi \circ \tilde{f}=\sigma \circ \pi $ and 
$\tilde{f}^{-1}(J(\tilde{f}))=J(\tilde{f})=\tilde{f}(J(\tilde{f}))$ 
(\cite{MR1767945}). 
We also remark that  by Zorn's lemma, there always exists a minimal set of $G$. 

For the fundamental tools and recent results of complex dynamics, see 
\cite{s11random, S13Coop}. 

\section{System of Functional Equations and Estimates}
\label{Functional}
In this section, 
we describe the system of functional equations for the elements of 
${\mathcal C}$ and we estimate the growth order of components of associated matrix cocycles $A(\omega ,k).$ 
More precisely, in Lemma~\ref{lem:matrix-growth} we show that 
every component of $A(\omega ,k)$ is of polynomial order with 
respect to $k$. Also, in some special cases we determine the precise 
polynomial growth rate.  

Let $(f_{1},\ldots, f_{s+1})\in (\mbox{Rat})^{s+1}$ be an element 
satisfying assumptions (1)(2)(3) of this paper and let $\vec{p}=(p_{i})_{i=1}^{s}\in (0,1)^{s}$ with $\sum _{i=1}^{s}p_{i}<1.$ 
Let $p_{s+1}=1-\sum _{i=1}^{s}p_{i}.$ 
%We denote by $C(\Chat)$ the Banach space of complex-valued functions
%on $\Chat$ endowed with the supremum norm. For a probability vector
%$(p_{1},\dots,p_{s+1})$ we define the transition operator $M$ given
%by 
Recall that the transition operator $M:C(\Chat )\rightarrow C(\Chat)$ of 
the random dynamical system generated by $(f_{1},\ldots, f_{s+1})$ and 
$\vec{p}$ in Section~\ref{Introduction} is defined by 
$M(h):=\sum_{i=1}^{s+1}p_{i}\cdot(h\circ f_{i}),\ h\in C(\Chat).$ 
Recall from \cite{s11random} that $M(C_{0})=C_{0}$. \foreignlanguage{british}{Next
lemma gives a system of functional equations for the elements of $\T$. }
\selectlanguage{british}%
\begin{lem}
\label{lem:cn-functionalequation}
For every \foreignlanguage{english}{\textup{$\vec{n}=(n_{i})_{i=1}^{s}\in\N_{0}^{s}$}}
we have 
\begin{equation}
C_{\vec{n}}=M(C_{\vec{n}})+\sum_{i=1}^{s}n_{i}\left(C_{\vec{n}-\vec{e}_{i}}\circ f_{i}-C_{\vec{n}-\vec{e}_{i}}\circ f_{s+1}\right), \label{Cn-functionalequation}
\end{equation}
where $\vec{e}_{i}$ denotes the element of $\N_{0}^{s}$ such that 
the $i$-th component is $1$ and all the other components are $0$.  
\end{lem}
\begin{proof}
The proof is by induction on the order $n:=|\vec{n}|\ge0$. The case
$n=0$ follows because $C_{0}=T_{\vec{p}}$ is a fixed point of $M$. Now suppose
that the lemma holds for derivatives of order $n\ge0$. Let $j\in\left\{ 1,\dots,s\right\} $.
By taking the partial derivative with respect to $p_{j}$ on both
sides of (\ref{Cn-functionalequation}) we see that 
\begin{eqnarray*}
C_{\vec{n}+\vec{e}_{j}} & = & M\left(C_{\vec{n}+\vec{e}_{j}}\right)+C_{\vec{n}}\circ f_{j}-C_{\vec{n}}\circ f_{s+1}+\sum_{i=1}^{s}n_{i}\left(C_{\vec{n}-\vec{e}_{i}+\vec{e}_{j}}\circ f_{i}-C_{\vec{n}-\vec{e}_{i}+\vec{e}_{j}}\circ f_{s+1}\right)\\
 & = & M\left(C_{\vec{n}+\vec{e_{j}}}\right)+\left(n_{j}+1\right)\left(C_{\vec{n}}\circ f_{j}-C_{\vec{n}}\circ f_{s+1}\right)+\sum_{i=1,i\neq j}^{s}n_{i}\left(C_{\vec{n}-\vec{e}_{i}+\vec{e}_{j}}\circ f_{i}-C_{\vec{n}-\vec{e}_{i}+\vec{e}_{j}}\circ f_{s+1}\right).
\end{eqnarray*}
Hence, the equation (\ref{Cn-functionalequation}) holds for $\vec{n}+\vec{e}_{j}$
and the lemma follows by induction on $n$. \end{proof} 
In the following, any element $A\in \R ^{\N_{0}^{s}\times \N_{0}^{s}}$ is 
represented as $A=(A_{\vec{x},\vec{y}})_{(\vec{x},\vec{y})\in \N_{0}^{s}\times \N_{0}^{s}}$, 
where $A_{\vec{x},\vec{y}}\in \R $, 
  and 
such an element $A$ is called an ($\N_{0}^{s}$-)matrix.  
%For a matrix $A\in \C ^{\N_{0}^{s}\times \N_{0}^{s}}$ 
%and $(\vec{x},\vec{y})\in \N_{0}^{s}\times \N_{0}^{s}$, we denote by 
$A_{\vec{x},\vec{y}}$ is called the $(\vec{x},\vec{y})$-component of $A.$

\begin{defn}
\label{def:cocycle matrix}For $\omega\in I^{\N}$ we define the matrix
$A_{0}(\omega,1)\in\R^{\N_{0}^{s}\times\N_{0}^{s}}$ given by 
\[
A_{0}(\omega,1):=\begin{cases}
\sum_{\vec{n}\in\N_{0}^{s}}\big(p_{\omega_{1}}\1_{\vec{n},\vec{n}}+n_{\omega_{1}}\1_{\vec{n},\vec{n}-\vec{e}_{\omega_{1}}}\big) & ,\omega_{1}\neq s+1\\
\sum_{\vec{n}\in\N_{0}^{s}}\big(p_{\omega_{1}}\1_{\vec{n},\vec{n}}-\sum_{i=1}^{s}n_{i}\1_{\vec{n},\vec{n}-\vec{e}_{i}}\big) & ,\omega_{1}=s+1,
\end{cases}
\]
where $\1_{\vec{n},\vec{m}}\in\R^{\N_{0}^{s}\times\N_{0}^{s}}$ denotes
the matrix such that for every $(\vec{x},\vec{y})\in \N_{0}^{s}\times \N_{0}^{s}$, 
the $(\vec{x},\vec{y})$-component $(\1_{\vec{n},\vec{m}})_{\vec{x},\vec{y}}$ 
of $\1 _{\vec{n},\vec{m}}$ 
is given by 
\begin{equation} \label{1nm}
(\1_{\vec{n},\vec{m}})_{\vec{x},\vec{y}}=\begin{cases}
1 & ,\vec{n}=\vec{x}\quad\mbox{and}\quad\vec{m}=\vec{y},\\
0 & ,\mbox{else}.
\end{cases}
\end{equation}
For $\omega\in I^{\N}$ and $k\in\N$ we define the matrix $A_{0}(\omega,k)\in\R^{\N_{0}^{s}\times\N_{0}^{s}}$
given by 
\[
A_{0}(\omega,k):=A_{0}(\omega,1)A_{0}(\sigma\omega,1)\cdots A_{0}(\sigma^{k-1}\omega,1)\in\R^{\N_{0}^{s}\times\N_{0}^{s}},
\]
where the matrix product $A_{0}(\tau,1)\cdot A_{0}(\upsilon,1)\in\R^{\N_{0}^{s}\times\N_{0}^{s}}$
is for $\tau,\upsilon\in I^{\N}$ and $\vec{l},\vec{m}\in\N_{0}^{s}$
given by 
\begin{equation}
\left(A_{0}(\tau,1)\cdot A_{0}(\upsilon,1)\right)_{\vec{l},\vec{m}}:=\sum_{\vec{k}\in\N_{0}^{s}}\left(A_{0}(\tau,1)\right)_{\vec{l},\vec{k}}\cdot\left(A_{0}(\upsilon,1)\right)_{\vec{k},\vec{m}}.\label{eq:definition-matrix product}
\end{equation}
Moreover, let $p_{\omega_{|k}}:=p_{\omega_{1}}p_{\omega_{2}}\cdot\dots\cdot p_{\omega_{k}}$
and define 
\[
A(\omega,k):=(p_{\omega_{|k}})^{-1}A_{0}(\omega,k)\in\R^{\N_{0}^{s}\times\N_{0}^{s}}.
\]
Also, for 
%$\omega\in I^{\N}$ and 
$a,b\in\Chat$ we define 
\[
U(a,b):=\left(u_{\vec{n}}(a,b)\right)_{\vec{n}\in\N_{0}^{s}}\in\R^{\N_{0}^{s}},\quad\mbox{where\quad}u_{\vec{n}}(a,b):=C_{\vec{n}}(a)-C_{\vec{n}}(b).
\]

Finally, for $\vec{n},\vec{m}\in\N_{0}^{s}$ we write $\vec{n}\le\vec{m}$
if $n_{i}\le m_{i}$ for all $1\le i\le s$. \end{defn}
\begin{rem*}
Note that (\ref{eq:definition-matrix product}) in Definition \ref{def:cocycle matrix}
is well defined, since there exist only finitely many non-zero entries
in each row of the matrix $A_{0}(\tau ,1) \in\R^{\N_{0}\times\N_{0}^{s}}$.
In the following we will frequently make use of the product of matrices
with an infinite index set, which requires explanation. These matrix
products will always be well defined, since either the first factor
of the product possesses at most finitely many non-zero entries in
each row, or the second factor contains at most finitely many non-zero
entries in each column. \end{rem*}
%\vspace{-4mm} 

To state the next lemma, we introduce the following matrices. 
\begin{defn} \label{def:Bij} For $i \in \{1, \dots ,s\}$ we introduce the $\N_0^s$-matrix $D_i$ given by 
\[
D_i =\sum_{\vec{n}\in\N_{0}^{s}}n_i \1_{\vec{n},\vec{n}-\vec{e}_i}.
\]
\end{defn}

Next lemma shows that the matrix cocycle $A$ is commutative. 
\begin{lem}
\label{lem:commutativity} Let $k\in \N$ and $i_1, i_2, \dots, i_k \in \{ 1,\dots, s\}$. Put $t_l=\card\{{j\le k\mid i_j=l \}}$, $l=1,\dots, s$ and let $\vec{t}=(t_l)_{l\le s}\in \N_0^s$. Then 
for every $\vec{u}, \vec{v}\in \N_{0}^{s}$, we have 
\begin{equation} \label{eq:Bij}
(D_{i_1}\cdot \dots \cdot D_{i_k})_{\vec{u},\vec{v}}=\begin{cases}
\prod_{i=1}^s u_i \cdot (u_i-1) \dots \cdot (u_i-t_i+1), \mbox{ if } \vec{v}=\vec{u}-\vec{t}. \\
0, \mbox{ else}.
\end{cases}
\end{equation}
Here, we put $u_{i}(u_{i}-1)\cdots (u_{i}-t_{i}+1)=1$ if $t_{i}=0.$
In particular, the matrices $(D_i)_{i=1,\dots, s}$ commute. Moreover, for all $\omega,\tau\in I^{\N}$ we have
\[
A(\omega,1)A(\tau,1)=A(\tau,1)A(\omega,1)\quad\mbox{and}\quad A_{0}(\omega,1)A_{0}(\tau,1)=A_{0}(\tau,1)A_{0}(\omega,1).
\]
\end{lem}
\begin{proof}
We only consider the case when $k=2$. The general case is left to the reader. Let $i,j\in \{1\dots s \}$. 
The following calculation proves (\ref{eq:Bij}). 
See  (\ref{1nm}) for the definition of $\1 _{\vec{n},\vec{m}}$. 
We have 
%By definition of $\1_{\vec{n},\vec{m}}$ (see (\ref{1nm})) we have  
\begin{eqnarray*}
(D_{i}\cdot D_{j})_{\vec{u},\vec{v}} &=& \left(\left(  \sum_{\vec{n}\in\N_{0}^{s}}n_i \1_{\vec{n},\vec{n}-\vec{e}_i} \right) \left(\sum_{\vec{n}\in\N_{0}^{s}}n_j\1_{\vec{n},\vec{n}-\vec{e}_j} \right)\right)_{\vec{u},\vec{v}} 
= \sum_{\vec{r}\in \N_0^s}\left(  \sum_{\vec{n}\in\N_{0}^{s}}n_i\1_{\vec{n},\vec{n}-\vec{e}_i} \right)_{\vec{u},\vec{r}} \left(  \sum_{\vec{n}\in\N_{0}^{s}}n_j \1_{\vec{n},\vec{n}-\vec{e}_j} \right)_{\vec{r},\vec{v}} \\
&=& \begin{cases}u_i\cdot (\vec{u}-\vec{e}_i)_j, \mbox{ if } \vec{v}=\vec{u}-\vec{e}_i -\vec{e}_j \\
0, \mbox{ else.}
\end{cases} 
\end{eqnarray*}
We see from (\ref{eq:Bij}) that the matrices $(D_i)_i$ commute. By  the definition of  $A_0$ we have $A_0(\omega,1)=p_{\omega_1}\id + D_{\omega_1}$, if $\omega_1\neq s+1$, where $\id=\sum\1_{\vec{n},\vec{n}}$, and $A_0(\omega,1)=p_{s+1} \id - \sum_{i=1}^s  D_i$, if $\omega_1=s+1$. Consequently, the commutativity of 
$A_{0}(\omega,1)$ and $A_{0}(\tau,1)$ follows. Thus, the commutativity of $A(\omega,1)$ and $A(\tau,1)$ follows. The proof is complete.
\end{proof}

The following lemma is easy to show by using the definition of 
$A(\omega ,k)$  and  induction on $k$ (see also the argument 
in the proof of Lemma~\ref{lem:commutativity}).   
\begin{lem}
\label{lem:idnm}
Let $\omega \in I^{\N}$ and  $k\in \N .$ Then 
$A(\omega ,k)_{\vec{n},\vec{n}}=1$ for every $\vec{n}\in \N_{0}^{s}.$ 
Also, $A(\omega ,k)_{\vec{n},\vec{m}}=0$ unless $\vec{m}\le \vec{n}.$ 
\end{lem}
The following lemma is easy to see by assumption (2) of our paper.  
\begin{lem}
\label{lem:e0}
There exists $\epsilon _{0}>0$ such that 
if $z\in f_{i}^{-1}(J(G))$ and $ j\neq i$ then $f_{j}(B(z,\epsilon _{0}))$ is included in a 
connected component of $F(G).$ 
\end{lem}
In the following, we fix an element $\epsilon _{0}>0$ given in Lemma~\ref{lem:e0}.
\begin{lem}
\label{lem:matrix-formulas}Let $\omega\in I^{\N}$, $z\in J_{\omega}$
and $k_{0}\in\N$. 
Let $a,b\in \Chat $ and suppose that 
$f_{\omega |_{k}}(a),f_{\omega |_{k}}(b)\in B(f_{\omega |_{k}}(z),\epsilon _{0})$ for 
all $0\leq k\leq k_{0}-1$, where $f_{\omega |_{0}}=id.$ Then 
%Then there exists $\epsilon>0$ such that for all
%$a,b\in B(z,\epsilon)$, 
\[
U(a,b)=A_{0}(\omega,k)U(f_{\omega_{|k}}\left(a\right),f_{\omega_{|k}}\left(b\right)),\quad1\le k\le k_{0}.
\]
That is, for each $\vec{n}\in\N_{0}^{s}$ we have 
\begin{eqnarray*}
u_{\vec{n}}(a,b) & = & \,\,\,\,\,\,\,\,\sum_{\vec{m}\in\N_{0}^{s}}A_{0}(\omega,k)_{\vec{n},\vec{m}}\,\, u_{\vec{m}}(f_{\omega_{|k}}\left(a\right),f_{\omega_{|k}}\left(b\right))\\
 & = & \sum_{\vec{m}\in\N_{0}^{s}:\vec{m}\le\vec{n}}A_{0}(\omega,k)_{\vec{n},\vec{m}}\,\, u_{\vec{m}}(f_{\omega_{|k}}\left(a\right),f_{\omega_{|k}}\left(b\right)).
\end{eqnarray*}
Moreover, if $u_{\vec{0}}(a,b)\neq0$ then
\[
\left(u_{\vec{0}}(a,b)\right)^{-1}U(a,b)=\left(u_{\vec{0}}(f_{\omega_{|k}}\left(a\right),f_{\omega_{|k}}\left(b\right))\right)^{-1}A(\omega,k)U(f_{\omega_{|k}}\left(a\right),f_{\omega_{|k}}\left(b\right)).
\]
\end{lem}
\begin{proof}
To prove the first assertion, it suffices to consider $k=1$. Then
general case then follows by induction on $k$. By Lemma \ref{lem:cn-functionalequation}
we have for $\vec{n}\in\N_{0}^{s}$, 
\begin{eqnarray*}
u_{\vec{n}}(a,b) & = & C_{\vec{n}}(a)-C_{\vec{n}}(b)\\
 & = & M(C_{\vec{n}})(a)-
 M(C_{\vec{n}})(b)+\sum_{i=1}^{s}n_{i}\left(C_{\vec{n}-\vec{e_{i}}}(f_{i}(a))-C_{\vec{n}-\vec{e_{i}}}(f_{i}(b))\right)-\\
 &  & \hspace{3cm}\sum_{i=1}^{s}n_{i}\left(C_{\vec{n}-\vec{e_{i}}}(f_{s+1}(a))-C_{\vec{n}-\vec{e_{i}}}(f_{s+1}(b))\right).
\end{eqnarray*}
Now first suppose that $\omega_{1}\neq s+1$. Since $C_{0}$ and hence
all its partial derivatives $C\in\T$ are  locally constant on $F(G)$ (see \cite[Theorem 3.15 (1)]{s11random}), 
by the choice of $\epsilon _{0}$, we have 
\begin{eqnarray*}
u_{\vec{n}}(a,b) & = & p_{\omega_{1}}\left(C_{\vec{n}}(f_{\omega_{1}}(a))-C_{\vec{n}}(f_{\omega_{1}}(b))\right)+n_{\omega_{1}}(C_{\vec{n}-\vec{e}_{\omega_{1}}}(f_{\omega_{1}}(a))-C_{\vec{n}-\vec{e}_{\omega_{1}}}(f_{\omega_{1}}(b)))\\
 & = & \left(A_{0}(\omega,1)U(f_{\omega_{1}}(a),f_{\omega_{1}}(b))\right)_{\vec{n}}.
\end{eqnarray*}
Similarly, if $\omega_{1}=s+1$ then we have 
\begin{align*}
u_{\vec{n}}(a,b) & =p_{\omega_{1}}\left(C_{\vec{n}}(f_{\omega_{1}}(a))-C_{\vec{n}}(f_{\omega_{1}}(b))\right)-\sum_{i=1}^{s}n_{i}\left(C_{\vec{n}-\vec{e_{i}}}(f_{s+1}(a))-C_{\vec{n}-\vec{e_{i}}}(f_{s+1}(b))\right)\\
 & =\left(A_{0}(\omega,1)U(f_{s+1}(a),f_{s+1}(b))\right)_{\vec{n}}. 
\end{align*}
The second assertion follows from the first by using 
$
u_{\vec{0}}(a,b)=p_{\omega_{|k}}u_{\vec{0}}(f_{\omega_{|k}}\left(a\right),f_{\omega_{|k}}\left(b\right)).
$
\end{proof}
We now prove the key lemma in which we estimate 
the polynomial growth order of the components of $A(\omega, k)$ 
as $k\rightarrow \infty .$  
\begin{lem}
\label{lem:matrix-growth}Let $\omega\in I^{\N}$ and $k\in\N$. Put
$m_{i}:=m_{i}(k):=\card\left\{ 1\le j\le k:\omega_{j}=i\right\} $
for $1\le i\le s+1$. 
Let $\vec{m}=(m_{i})_{i=1}^{s}\in \N_{0}^{s}.$ 
Let $\vec{q}$, $\vec{r}\in \N_0^s$ with  $\vec{0}\le\vec{r}\le\vec{q}$. Then we have 
\begin{align*} 
A(\omega,k)_{\vec{q},\vec{r}} & =\sum_{\substack{\vec{q}-\vec{r}-\vec{m}\le \vec{t}\le\vec{q}-\vec{r}\\
0\le |\vec{t}|\le m_{s+1}
}
}\left((-1)^{|\vec{t}|}p_{s+1}^{-|\vec{t}|}{m_{s+1} \choose |\vec{t}|}|\vec{t}|!\prod_{i=1}^{s}{m_{i} \choose q_{i}-r_{i}-t_{i}}\frac{q_{i}!}{(t_{i})!r_{i}!p_{i}^{q_{i}-r_{i}-t_i}}\right)\\
 & =\sum_{\substack{\vec{q}-\vec{r}-\vec{m}\le \vec{t}\le\vec{q}-\vec{r}\\
0\le |\vec{t}|\le m_{s+1}
}
}\left((-1)^{|\vec{t}|}p_{s+1}^{-|\vec{t}|}\frac{m_{s+1}!}{\left(m_{s+1}-|\vec{t}|\right)!}\prod_{i=1}^{s}\frac{m_{i}!q_{i}!}{\left(m_{i}-\left(q_{i}-r_{i}-t_{i}\right)\right)!\left(q_{i}-r_{i}-t_{i}\right)!t_{i}!r_{i}!p_{i}^{q_{i}-r_{i}-t_i}}\right),
\end{align*}
where $\vec{t}=(t_i)_{1\le i \le s}$. 
In particular, there exists a constant $K\ge1$
which depends on $\vec{q}$ and the probability vector $\vec{p}$ 
but not on $k$ 
such that 
\[
\left|A(\omega,k)_{\vec{q},\vec{r}}\right|\le K\left(\prod_{i=1}^{s}\tilde{m}_{i}^{q_{i}-r_{i}}\right)\tilde{m}_{s+1}^{|\vec{q}|-|\vec{r}|}
 \mbox{ and }
\left|A(\omega ,k)_{\vec{q},\vec{r}}\right| \le Kk^{|\vec{q}|},
\]
where  $\tilde{m}_{j}:=\max\{ 1, m_{j}\}$ for $1\leq j\leq s+1.$
If $\omega_{j}\neq s+1$ for all $1\leq j\leq k$ 
and $m_{i}>q_{i}-r_{i}$ for all $1\le i\le s$,
then there exists $K'>0$ depending only on $\vec{q}$  such that
\[
A(\omega,k)_{\vec{q},\vec{r}}\ge K'\prod_{i=1}^{s}m_{i}^{q_{i}-r_{i}}.
\]
\end{lem}
\vspace{-4.6mm} 
\begin{proof}
By Lemma \ref{lem:commutativity} we have 
\begin{equation} \label{eq:Acombi}
A(\omega,k)  =  \prod_{i=1}^{s}\left(\id+ p_i^{-1} D_i \right)^{m_{i}}\left(\id-p_{s+1}^{-1}\sum_{i=1}^{s}D_i\right)^{m_{s+1}}.
\end{equation}
To expand the right-hand side, we use the multinomial coefficient, which is given by ${{n} \choose {t_1, \,\, t_2, \,\, \ldots \,\, , t_s}}=\frac{n!}{\prod_{i=1}^{s}(t_{i})!}$ and which satisfies
$$ (x_1 + x_2 + \cdots +x_s)^n=\sum_{|\vec{t}|\le n} {{n} \choose {t_1,\,\,t_2\,\,\ldots, t_s}} x_1^{t_1}x_2^{t_2}\dots  x_s^{t_s}.$$ 
By (\ref{eq:Acombi}) and Lemma~\ref{lem:commutativity} 
we obtain, for each $\vec{q}$, $\vec{r}$ with  $\vec{0}\le\vec{r}\le\vec{q}$,  
\begin{eqnarray*}
(A(\omega,k))_{\vec{q},\vec{r}} & = &  \sum_{\substack{\vec{q}-\vec{r}-\vec{m}\le \vec{t}\le\vec{q}-\vec{r}\\
0\le |\vec{t}|\le m_{s+1}
}
} \sum_{\substack{ \mathcal{I}_{s+1} \subset \{ 1,\dots, m_{s+1}\} \\
|\mathcal{I}_{s+1}|= |\vec{t}|
}
} {|{\mathcal{I}_{s+1}}| \choose {t_1, \,\, t_2, \,\, \ldots \,\, , t_s}} 
(-1)^{|\vec{t}|}p_{s+1}^{-|\vec{t}|}\prod_{i=1}^{s}{m_{i} \choose q_{i}-r_{i}-t_{i}}\frac{q_i !}{r_i ! p_{i}^{q_{i}-r_{i}-t_i}}.
\end{eqnarray*}
%Here, we put $q_{i}(q_{i}-1)\cdots (r_{i}+1)=1$ if $q_{i}=r_{i}.$ 
Note  that, to deduce the above formula, when we expand  the term $\left(\id-p_{s+1}^{-1}\sum_{i=1}^{s}D_i\right)^{m_{s+1}}$ on the right hand side of (\ref{eq:Acombi}), 
for any $\vec{t}$ with $\vec{q}-\vec{r}-\vec{m}\leq 
\vec{t}\leq \vec{q}-\vec{r}$, $0\leq |\vec{t}|\leq 
m_{s+1}$,  and 
for any subset $\mathcal{I}_{s+1}\subset \{ 1,\dots,m_{s+1}\}$ with $|\mathcal{I}_{s+1}|=|\vec{t}|$,  we picked the factor $-p_{s+1}^{-1}\sum_{i=1}^{s}D_i$  for any element $j\in \mathcal{I}_{s+1}$,  and  we picked the identity  for any element $j\in \{1,\dots, m_{s+1} \} \setminus \mathcal{I}_{s+1}$.
Finally, a simple calculation finishes the proof of the first assertion of the lemma.

For the upper bound of $\left|A(\omega,k)_{\vec{q},\vec{r}}\right|$
we observe that with some constant \foreignlanguage{english}{$K_{0}$ which depends on 
$\vec{q}$ and the probability vector $\vec{p}$ 
but not on $k$ 
we have} 
\begin{eqnarray*}
p_{s+1}^{-|\vec{t}|}\frac{m_{s+1}!}{\left(m_{s+1}-|\vec{t}|\right)!}\prod_{i=1}^{s}\frac{m_{i}!q_{i}!}{\left(m_{i}-\left(q_{i}-r_{i}-t_{i}\right)\right)!\left(q_{i}-r_{i}-t_{i}\right)!t_{i}!r_{i}!p_{i}^{q_{i}-r_{i}-t_{i}}} & \le & K_{0}\tilde{m}_{s+1}^{|\vec{t}|}\prod_{i=1}^{s}\tilde{m}_{i}^{q_{i}-r_{i}-t_{i}}\\
 & \le & K_{0}\left(\prod_{i=1}^{s}\tilde{m}_{i}^{q_{i}-r_{i}}\right)\tilde{m}_{s+1}^{|\vec{q}|-|\vec{r}|}.
\end{eqnarray*}
%\vspace{-4mm} 
Since $\sum_{i=1}^{s+1}m_{i}=k$ 
we see that 
$\tilde{m}_{s+1}^{|\vec{t}|}
\prod _{i=1}^{s}\tilde{m}_{i}^{q_{i}-r_{i}-t_{i}}\leq k^{|\vec{q}|}. $ 
%and $q_{i}\le|\vec{q}|$ we see that
%$\left(\prod_{i=1}^{s}m_{i}^{q_{i}-r_{i}}\right)m_{s+1}^{|\vec{q}|-|\vec{r}|}\le %k^{|\vec{q}|}$.

Now suppose that $\omega_{j}\neq s+1$ for all $1\leq j\leq k$  
and $m_{i}>q_{i}-r_{i}$ 
  for $1\le i\le s$. Then we
have 
\[
A(\omega,k)_{\vec{q},\vec{r}}=\prod_{i=1}^{s}\frac{m_{i}!q_{i}!}{\left(m_{i}-\left(q_{i}-r_{i}\right)\right)!\left(q_{i}-r_{i}\right)!r_{i}!p_{i}^{q_{i}-r_{i}}}.
\]
Clearly, 
%we have 
with some constant $K'>0$ which depends only on 
$\vec{q}$ we have 
%and the probability vector $\vec{p}$  
that
\begin{eqnarray*}
\frac{m_{i}!q_{i}!}{\left(m_{i}-\left(q_{i}-r_{i}\right)\right)!\left(q_{i}-r_{i}\right)!r_{i}!p_{i}^{q_{i}-r_{i}}} & \ge & K'm_{i}^{q_{i}-r_{i}},
\end{eqnarray*}
which finishes the proof of the lower bound.\end{proof}
\selectlanguage{english}%
\vspace{-4mm} 
\begin{lem}
\label{lem:growth of solutions}Let $x_{0}\in J(G)$ and let $\epsilon>0$.
\foreignlanguage{british}{Let $\vec{n}\in\N_{0}^{s}$ and set $n:=|\vec{n}|$.}
Then there exists a constant $K>0$ such that for every $k\in\N$
there exist points $a_{k}\in (B(x_{0},\epsilon)\cap J(G))\setminus\left\{ x_{0}\right\} $
and $b_{k}\in B(x_{0},\epsilon)\setminus\left\{ x_{0}\right\} $ with
$u_{\vec{0}}(a_{k},b_{k})\neq0$ such that for $0\le\vec{q}\le\vec{n}$,
\[
K^{-1}k^{\sum_{i=1}^{s}q_{i}\left(n+1\right)^{i-1}}\le\frac{u_{\vec{q}}(a_{k},b_{k})}{u_{\vec{0}}(a_{k},b_{k})}\le Kk^{\sum_{i=1}^{s}q_{i}\left(n+1\right)^{i-1}}.
\]
\end{lem}
\selectlanguage{british}%
\vspace{-5mm} 
\begin{proof}
By the density of the repelling fixed points in $J(G)$ 
(\cite[Theorem 3.1]{MR1397693})  
there exist 
$z_{0}\in B(x_{0},\epsilon)$ and $g\in G$ such that $g(z_{0})=z_{0}$
and \foreignlanguage{english}{$|g'(z_{0})|>1$. Since $\deg(g)\ge2$
we have }$E(g)\subset P(g)\subset P(G)$, where\foreignlanguage{english}{
}$E(g)=E(\langle g\rangle )$ denotes the set of exceptional points of $g$.
Since $G$ is hyperbolic we have $J(G)\subset\Chat\setminus P(G)\subset\Chat\setminus E(g)$.
We may assume that $g(B(z_{0},\epsilon))\supset B(z_{0},\epsilon)$.
Moreover, we have $\bigcup_{n\in\N}g^{n}\left(B(z_{0},\epsilon)\right)=\Chat\setminus E(g)$.
Hence, there exists $n\in\N$ such that $J(G)\subset g^{n}\left(B(z_{0},\epsilon)\right)$.
We may assume that $n=1$ and $J(G)\subset g(B(x_{0},\epsilon))$. 

Since $C_{0}$ is not locally constant on any neighborhood of any point of $J(G)$ (see \cite[Lemma 3.75]{s11random}) and 
since 
$J(G)$ is an uncountable perfect set 
(see \cite[Lemma 3.1]{MR1397693}),  
there exist $a\in J(G)\setminus G(x_{0})$
and $b\in\Chat\setminus G(x_{0})$ close to $a$ such that $C_{0}(a)\neq C_{0}(b)$. 
%and $C_{0}(g(a))\neq C_{0}(g(b))$. 
%To see this, first recall that
%and then observe that $C_{0}(g(a))\neq C_{0}(g(b))$ follows from
%$C_{0}(a)\neq C_{0}(b)$ because $MC_{0}=C_{0}$ and $C_{0}$ is locally
%constant on $F(G)$. 

For each $k\in\N$ and $1\le i\le s$ we set $m_{i}(k):=k^{\left(\left(n+1\right)^{i-1}\right)}$.
Then we define $h_{k}:=f_{1}^{m_{1}(k)}\circ \cdots \circ f_{s}^{m_{s}(k)}$. 
Since $G$ is hyperbolic, we have $P(G)\subset F(G).$ 
For each connected component $U$ of $F(G)$, we take the hyperbolic 
metric on $U.$ 
For each connected component $U$ of $F(G)$ with 
$U\cap P(G)\neq \emptyset $, let 
$B_{h}(P(G)\cap U, 1)$ be the $1$-neighborhood of $P(G)\cap U$ 
in $U$ with respect to the hyperbolic metric on $U.$ 
Let $V=\cup B_{h}(P(G)\cap U, 1)$, where the union is taken 
over all connected components $U$ of $F(G)$ with $U\cap P(G)\neq 
\emptyset .$ Then $G(V)\subset V$, $\overline{V}\subset F(G)$ and $J(G)\subset \Chat 
\setminus \overline{V}.$  
Since $a\in J(G)$, there exist $\eta>0$
and a holomorphic inverse branch $\gamma_{k}:B(a,\eta)\rightarrow\Chat$
such that $h_{k}\circ\gamma_{k}=\id_{B(a,\eta)}$, for each $k\in\N$.
We may assume that $b\in B(a,\eta)$. Put $\tilde{a}_{k}=\gamma_{k}(a)\in J(G)$
and $\tilde{b}_{k}=\gamma_{k}(b)$. 
Since $G(V)\subset V$, we have that 
$(\gamma_{k})_{k\in\N}$
is normal in $B(a,\eta ).$ Thus we may assume that $d(\tilde{a}_{k},\tilde{b}_{k})\le\delta$
for all $k\in\N$, where $\delta >0$ 
is a small number. Since $\tilde{a}_{k}\in J(G)$, there exists $a_{k}\in J(G)\cap B(x_{0},\epsilon)$
with $g(a_{k})=\tilde{a}_{k}$ for all $k\in \N.$ 
%By making $\delta$ sufficiently small,
%we can find $b_{k}\in B(x_{0},\epsilon)$ with $g(b_{k})=\tilde{b}_{k}$. 
We write $g=f_{\tau_{r}}\circ\dots\circ f_{\tau_{1}}$ for some $r\in\N$
and $\tau=(\tau_{1},\dots,\tau_{r})\in I^{r}.$ 
By making $\delta$ sufficiently small, for each $k\in \N$ 
let $\alpha_{k}:\gamma_{k}(B(a,\eta ))\rightarrow \Chat$ 
be the holomorphic map such that 
$g\circ \alpha_{k}=\id_{\gamma_{k}(B(a,\eta) )}$ and 
$\alpha_{k}(\tilde{a}_{k})=a_{k}.$ 
We may assume that 
$\alpha_{k}(\gamma_{k}(B(a,\eta )))\subset B(x_{0},\epsilon )$ 
and $\diam f_{\tau_{j}}\circ \cdots \circ f_{\tau _{1}}
(\alpha_{k}(\gamma_{k}(B(a,\eta ))))<\epsilon_{0}$ for all $j=0,\ldots, r$
 where for  $j=0$ we set $f_{\tau_{j}}\circ \cdots \circ f_{{\tau }_{1}}=\id.$
Let $b_{k}=\alpha_{k}(\tilde{b}_{k})\in B(x_{0},\epsilon )$.
Put $\overline{\tau}:=\left(\tau_{1},\dots,\tau_{r},\tau_{1},\dots,\tau_{r},\dots\right)\in I^{\N}$. 
Since $M(C_{0})=C_{0}$, $C_{0}$ is locally constant on $F(G)$ (\cite[Theorem 3.15 (1)]{s11random}) 
and $C_{0}(a)-C_{0}(b)\neq 0$,  
if $\delta $ is small enough, then Lemma~\ref{lem:e0} implies that 
$u_{\vec{0}}(a_{k},b_{k})=C_{0}(a_{k})-C_{0}(b_{k})\neq 0$ and 
$u_{\vec{0}}(\tilde{a}_{k},\tilde{b}_{k})=C_{0}(\tilde{a}_{k})-C_{0}(\tilde{b}_{k})\neq 0.$ 
 Since $g(a_{k})=\tilde{a}_{k}$ and $g(b_{k})=\tilde{b}_{k}$, Lemma
\ref{lem:matrix-formulas}  and $J(G)=\bigsqcup _{\omega \in I^{\N}}J_{\omega }$
 (Lemma ~\ref{lem:exp}) yield 
\begin{equation}
\label{eq:u0akbk1}
\left(u_{\vec{0}}(a_{k},b_{k})\right)^{-1}U(a_{k},b_{k})=\left(u_{\vec{0}}(\tilde{a}_{k},\tilde{b}_{k})\right)^{-1}A(\overline{\tau},r)U(\tilde{a}_{k},\tilde{b}_{k}).
\end{equation}
Put $\xi_{k}:=(1^{m_{1}(k)},2^{m_{2}(k)},\dots,s^{m_{s}(k)})\in I^{\sum_{i=1}^{s}m_{i}(k)}$
where $u^{m}:=(u,u,\dots,u)\in I^{m}$ for $u\in\left\{ 1,\dots,s+1\right\} $.
Since $h_{k}(\tilde{a}_{k})=a$ and $h_{k}(\tilde{b}_{k})=b$, we
have by Lemmas \ref{lem:matrix-formulas} and \ref{lem:commutativity}, 
\begin{equation}
\label{eq:u0akbk2}
\left(u_{\vec{0}}(\tilde{a}_{k},\tilde{b}_{k})\right)^{-1}U(\tilde{a}_{k},\tilde{b}_{k})=\left(u_{\vec{0}}(a,b)\right)^{-1}A\left(\overline{\xi_{k}},\sum_{i=1}^{s}m_{i}(k)\right)U(a,b).
\end{equation}

By combining the previous two equalities 
(\ref{eq:u0akbk1}) (\ref{eq:u0akbk2}) we have
\[
\left(u_{\vec{0}}(a_{k},b_{k})\right)^{-1}U(a_{k},b_{k})=\left(u_{\vec{0}}(a,b)\right)^{-1}A(\overline{\tau},r)
A\left(\overline{\xi_{k}},\sum_{i=1}^{s}m_{i}(k)\right)U(a,b).
\]

Since $h_{k}\in\left\langle f_{1},\dots,f_{s}\right\rangle $, it
follows from Lemma \ref{lem:matrix-growth} that for $\vec{q}\le\vec{n}$,
\begin{eqnarray*}
\left(A\left(\overline{\xi_{k}},\sum_{i=1}^{s}m_{i}(k)\right)U(a,b)\right)_{\vec{q}} & \asymp & \prod_{i=1}^{s}(m_{i}(k))^{q_{i}}\asymp k^{\sum_{i=1}^{s}q_{i}\left(n+1\right)^{i-1}} 
\mbox{ as }k\rightarrow \infty ,
\end{eqnarray*}
where for any two non-negative functions $\phi _{1}(k)$ and $\phi _{2}(k)$ 
of $k\in \Bbb{N}$, we write  
$\phi _{1}(k)\asymp \phi _{2}(k)$ as $k\rightarrow \infty $ 
if there exists a constant $D>1$ such that 
$D^{-1}\phi _{2}(k)\leq \phi _{1}(k)\leq D\phi _{2}(k)$ for every 
$k\in  \Bbb{N}.$ 
% Here we used that for $\vec{t}=0$ the maximal growth rate is obtained
%and $A(\overline{\xi_{k}},r)_{\vec{q},\vec{q}}=1$. 
Also by Lemma~\ref{lem:idnm}, we have 
$$\left(A(\overline{\tau },r)A\left(\overline{\xi _{k}},\sum _{i=1}^{s}m_{i}(k)\right)U(a,b)\right)_{\vec{q}}
=\sum _{\vec{r}\leq \vec{q}}
A(\overline{\tau },r)_{\vec{q},\vec{r}}\left(A\left(\overline{\xi _{k}},\sum _{i=1}^{s}m_{i}(k)\right)U(a,b)\right)_{\vec{r}}$$  
and $A(\overline{\tau },r)_{\vec{q},\vec{q}}=1$. 
The proof is complete.\end{proof}
\begin{lem}
\label{detlemma}Let $m\in\N$ and $0\le w_{1}<w_{2}<\dots<w_{m}$
be natural numbers and let $K\ge1$ be a constant. Then there exist
$0\le d_{1}<d_{2}<\dots<d_{m}$ and $\ell_{0}\in\N$ such that for
all $\ell\ge\ell_{0}$ and for all $B=(B_{ij})_{i,j}\in\R^{m\times m}$ satisfying
\[
K^{-1}\ell^{w_{i}d_{j}}\le B{}_{i,j}\le K\ell^{w_{i}d_{j}},\quad\mbox{for all }i,j\le m,
\]
we have 
\vspace{-3mm} 
\[
\det\left(B\right)\ge1.
\]
\end{lem}
\vspace{-5mm} 
\begin{proof}
The proof is by induction on $m\in\N$. Let $0\le w_{1}<w_{2}<\dots<w_{m}<w_{m+1}$.
By induction hypothesis there exist $0\le d_{1}<d_{2}<\dots<d_{m}$
and $\ell_{0}$ for the sequence $0\le w_{1}<w_{2}<\dots<w_{m}$.
Let $d_{m+1}\in\N$ and let $B=(B_{ij})_{i,j\leq m+1}\in\R^{(m+1)\times(m+1)}$ be a matrix
satisfying $K^{-1}\ell^{w_{i}d_{j}}\le B{}_{i,j}\le K\ell^{w_{i}d_{j}}$,
for each $i,j\le m+1$ and for all $\ell\ge\ell_{0}$. Put $B':=\left(B_{i,j}\right)_{i,j\le m}\in\R^{m\times m}$.
By the Laplace expansion of $\det(B)$ along the $(m+1)$th column,
we see that 
\begin{eqnarray*}
\det(B) & \ge & K^{-1}\ell^{w_{m+1}d_{m+1}}\det\left(B'\right)+O(\ell^{w_{m}d_{m+1}}(\ell^{w_{m+1}d_{m}})^{m}), \mbox{ as } \ell \rightarrow \infty . 
\end{eqnarray*}
Since 
\[
\frac{\ell^{w_{m}d_{m+1}}\ell^{m\cdot w_{m+1}d_{m}}}{\ell^{w_{m+1}d_{m+1}}}=\ell^{w_{m}d_{m+1}-w_{m+1}d_{m+1}+mw_{m+1}d_{m}}=\ell^{d_{m+1}(w_{m}-w_{m+1})+mw_{m+1}d_{m}},
\]
we see that, for $d_{m+1}$ sufficiently large, we have $\ell^{w_{m}d_{m+1}}\left(\ell^{w_{m+1}d_{m}}\right)^{m}\in o\left(\ell^{w_{m+1}d_{m+1}}\right)$
as $l$ tends to infinity. Since by our induction hypothesis, we have
$\det\left(B'\right)\ge1$ for $\ell\ge\ell_{0}$, the lemma follows.\end{proof}
\selectlanguage{english}%
\begin{prop}
\label{prop:fullrank}Let $x_{0}\in J(G)$ and let $\epsilon>0$.
\foreignlanguage{british}{Let $\vec{n}\in\N_{0}^{s}$.} Then there exist
 families $(a_{\vec{r}}')_{\vec{r}\le\vec{n}}$ and $(b_{\vec{r}}')_{\vec{r}\le\vec{n}}$
with $a_{\vec{r}}'\in (B(x_{0},\epsilon)\cap J(G))\setminus\left\{ x_{0}\right\} $,
$b_{\vec{r}}'\in B(x_{0},\epsilon)\setminus\left\{ x_{0}\right\} $
and $u_{\vec{0}}(a_{\vec{r}}',b_{\vec{r}}')\neq0$, for all $\vec{r}\le\vec{n}$,
such that the matrix 
\[
\left(u_{\vec{q}}(a_{\vec{r}}',b_{\vec{r}}')\right)_{\substack{\vec{r}\le\vec{n}\\
\vec{q}\le\vec{n}
}
}
\]
is invertible. \end{prop}
\begin{proof}
Put\foreignlanguage{british}{ $n:=|\vec{n}|$}. Define $\iota:\left\{ \vec{q}:\vec{q}\le\vec{n}\right\} \rightarrow\N$
given by $\iota\left(\vec{q}\right):=\sum_{i=1}^{s}q_{i}\left(n+1\right)^{i-1}$. 
By Lemma \ref{lem:growth of solutions} there exists a constant $K>0$
such that for every $k\in\N$ there exist points $a_{k}\in (B(x_{0},\epsilon)\cap J(G))\setminus\left\{ x_{0}\right\} $
and $b_{k}\in B(x_{0},\epsilon)\setminus\left\{ x_{0}\right\} $ with
$u_{\vec{0}}(a_{k},b_{k})\neq0$ such that for $0\le\vec{q}\le\vec{n}$,
\[
K^{-1}k^{\iota\left(\vec{q}\right)}\le\frac{u_{\vec{q}}(a_{k},b_{k})}{u_{\vec{0}}(a_{k},b_{k})}\le Kk^{\iota\left(\vec{q}\right)}.
\]
Since $q_{i}\le n$ we have that the numbers $\iota\left(\vec{q}\right)$,
$\vec{q}\le\vec{n}$, are pairwise distinct. We put the elements $\iota\left(\vec{q}\right),\vec{q}\le\vec{n}$
in increasing order and denote them by $w_{1}<w_{2}<\dots<w_{m}$,
where $m:=\card\left\{ \vec{q}:\vec{q}\le\vec{n}\right\} $. Let $d_{1}<\dots<d_{m}\in\N$
and $\ell_{0}\in\N$ be the elements given by Lemma \ref{detlemma}
for the sequence $w_{1}<w_{2}<\dots<w_{m}$ and the constant $K$.
For $\vec{r}\le \vec{n}$ we put $e\left(\vec{r}\right):=d_{\iota\left(\vec{r}\right)}$, $h\left(\vec{r}\right):=\ell_{0}^{e\left(\vec{r}\right)}$
and define 
\[
a_{\vec{r}}':=a_{h\left(\vec{r}\right)}.
\]
Hence, by Lemma \ref{detlemma} we have that $\left(\frac{u_{\vec{q}}(a_{\vec{r}}',b_{\vec{r}}')}{u_{\vec{0}}(a_{\vec{r}}',b_{\vec{r}}')}\right)_{\substack{\vec{r}\le\vec{n}\\
\vec{q}\le\vec{n}
}
}$ is invertible. The proof is complete. 
\end{proof}
\selectlanguage{british}%
\vspace{-5mm} 
\section{Proof of Theorem \ref{main-thm}}
\label{Proof}
In this section, we give the proof of Theorem~\ref{main-thm}. 
\vspace{-4mm} 
\subsection{Lower bound of the pointwise H\"older exponent }
\begin{lem}
\label{lem:lowerbound}Let $C=\sum_{\vec{n}}\beta_{\vec{n}}C_{\vec{n}}\in\T$
be non-trivial. Let $\omega\in I^{\N}$, $z\in J_{\omega}$ and $\vec{n}\in\N_{0}^{s}$.
Then 
\[
\Hol(\sum_{\vec{n}}\beta_{\vec{n}}C_{\vec{n}},z)\ge\liminf_{k\rightarrow \infty }\frac{S_{k}\tilde{\psi}\left(\omega,z\right)}{S_{k}\tilde{\varphi}\left(\omega,z\right)}.
\]
\end{lem}
\begin{proof}
%Since $f$ is expanding, we have that $G$ is hyperbolic and there
%exists a non-empty $G$-forward invariant compact subset of $F\left(G\right)$.
Let $V$ be a neighborhood of $P(G)$ in $F(G)$ as in the proof of 
Lemma~\ref{lem:growth of solutions}. Then $\overline{V}\subset F(G)$ and 
$G(V)\subset V.$  
%Hence, there exists $R>0$ such that, 
Let $R>0$ such that $B(J(G),R)\subset \Chat \setminus \overline{V}.$
 Then 
for each $k\in\N$, \foreignlanguage{english}{there
exists a holomorphic branch} $\phi_{k}:B\bigl(f_{\omega_{|k}}\left(z\right),R\bigr)\rightarrow\Chat$
of $f_{\omega_{|k}}^{-1}$ such that $f_{\omega_{|k}}\left(\phi_{k}\left(y\right)\right)=y$
for $y\in B\bigl(f_{\omega_{|k}}\left(z\right),R\bigr)$ and $\phi_{k}\bigl(f_{\omega_{|k}}\left(z\right)\bigr)=z$.
Since $G(V)\subset V$, 
for every $\epsilon>0$ there exists $r_{0}\le R$ such that, for
the sets $B_{k}$, which are for $k\in\N$ given by 
\[
B_{k}:=\phi_{k}\bigl(B\bigl(f_{\omega_{|k}}\left(z\right),r_{0}\bigr)\bigr),
\]
we have that $\diam\bigl(f_{\omega_{|j}}\left(B_{k}\right)\bigr)\le\epsilon$
for all $1\le j\le k$. 
%Put $n:=|\vec{n}|$ and 
Set $m_{i}(k)=\card\left\{ 1\le j\le k:\omega_{j}=i\right\} $
for $1\le i\le s$. 
Let $\vec{n}_{\max}\in \N_{0}^{s}$ be an element 
such that for every $\vec{n}\in \N_{0}^{s}$ with $\beta _{\vec{n}}\neq 0$, we have 
$\vec{n}\leq \vec{n}_{\max }.$ 
%:=\max\left\{ \vec{n}\mid\beta_{\vec{n}}\neq0\right\} $.
Taking $\epsilon >0$ such that $0<\epsilon <\epsilon _{0}$, 
by Lemma \ref{lem:matrix-formulas} and Lemma \ref{lem:matrix-growth}
there exists $K\ge1$ such that 
\begin{eqnarray*}
\sup_{y\in B_{k}}\left|C(y)-C(z)\right| & = & \sup_{y\in B_{k}}\left|\sum_{\vec{n}\le\vec{n}_{\max}}\beta_{\vec{n}}\left(C_{\vec{n}}(y)-C_{\vec{n}}(z)\right)\right|\\
 & = & p_{\omega_{|k}}\left(\sup_{y\in B_{k}}\left|\sum_{\vec{n}\le\vec{n}_{\max}}\beta_{\vec{n}}\sum_{\vec{j}\in\N_{0}^{s}:\vec{j}\le\vec{n}}A(\omega,k)_{\vec{n},\vec{j}}\cdot u_{\vec{j}}(f_{\omega_{|k}}(y),f_{\omega_{|k}}(z))\right|\right)\\
 & \le & p_{\omega_{|k}}\sum_{\vec{n}\le\vec{n}_{\max}}|\beta_{\vec{n}}|\card\left\{ \vec{j}:\vec{j}\le\vec{n}_{\max}\right\} Kk^{|\vec{n}_{\max}|}\cdot2\max_{\vec{j}\le\vec{n}_{\max}}\Vert C_{\vec{j}}\Vert.
\end{eqnarray*}
We have thus shown that 
\begin{equation}
\log\sup_{y\in B_{k}}\left|C(y)-C(z)\right|\le S_{k}\tilde{\psi}(\omega,z)+\log\left(\sum_{\vec{n}\le\vec{n}_{\max}}|\beta_{\vec{n}}|\card\left\{ \vec{j}:\vec{j}\le\vec{n}_{\max}\right\} Kk^{|\vec{n}_{\max}|}\cdot2\max_{\vec{j}\le\vec{n}_{\max}}\Vert C_{\vec{j}}\Vert\right).\label{eq:upperestimate}
\end{equation}
By \cite[Lemma 5.1]{JS13b} and Koebe's distortion theorem 
(see also the proof of \cite[Lemma 5.2]{JS13b})
we have 
\begin{equation}
\label{eq:holcz}
\Hol(C,z)=\liminf_{r\rightarrow 0}\frac{\log\sup_{y\in B(z,r)}\left|C(y)-C(z)\right|}{\log r}=\liminf_{k\rightarrow \infty }\frac{\log\sup_{y\in B_{k}}\left|C(y)-C(z)\right|}{S_{k}\tilde{\varphi}(\omega,z)}.
\end{equation}
Since $G$ is hyperbolic, \cite[Theorem 2.6]{MR1625124} implies that 
there exist $m\in \N $ and $\theta <0$ such that $S_{m}\tilde{\varphi }<\theta <0.$ 
Combining (\ref{eq:holcz}) with (\ref{eq:upperestimate}) and $S_{k}\tilde{\varphi }<0$ 
for every large $k$,   
 we see that 
\[
\Hol(C,z)\ge\liminf_{k\rightarrow \infty }\frac{S_{k}\tilde{\psi}(\omega,z)}
{S_{k}\tilde{\varphi}(\omega,z)}+\liminf_{k\rightarrow \infty }\frac{\log\left(\sum_{\vec{n}\le\vec{n}_{\max}}|\beta_{\vec{n}}|\card\left\{ \vec{j}:\vec{j}\le\vec{n}_{\max}\right\} Kk^{|\vec{n}_{\max}|}\cdot2\max_{\vec{j}\le\vec{n}_{\max}}\Vert C_{\vec{j}}\Vert\right)}{S_{k}\tilde{\varphi}(\omega,z)}.
\] 
%Since $G$ is hyperbolic, there exists $\theta<0$ such that $\tilde{\varphi}\le\theta<0$.
Consequently, we have that 
\[
\lim_{k\rightarrow \infty }\frac{\log\left(\sum_{\vec{n}\le\vec{n}_{\max}}|\beta_{\vec{n}}|\card\left\{ \vec{j}:\vec{j}\le\vec{n}_{\max}\right\} Kk^{|\vec{n}_{\max}|}\cdot2\max_{\vec{j}\le\vec{n}_{\max}}\Vert C_{\vec{j}}\Vert\right)}{S_{k}\tilde{\varphi}(\omega,z)}=0,
\]
which completes the proof of the lemma. 
\end{proof}

\subsection{Upper bound of the pointwise H\"older exponent }

To prove the upper bound of the point H\"older exponent, the following
lemma is crucial.
\begin{lem}
\label{lem:eta-keylemma}Let $\omega\in I^{\N}$ and $x_{0}\in J_{\omega}$.
Let $\vec{n}_{\max}\in\N_{0}^{s}$ and let $\left(\beta_{\vec{n}}\right)_{\vec{n}\le\vec{n}_{\max}}\neq0$. 
Let $(j(k))_{k\in \N}$ be a sequence of positive integers such that 
$j(k)\rightarrow \infty $ as $k\rightarrow \infty .$ 
Then for every $\epsilon>0$ there exist $a,b\in B(x_{0},\epsilon)\setminus\left\{ x_{0}\right\} $
with $a\neq b$ such that\foreignlanguage{english}{\textup{ 
\[
\eta:=\limsup_{k\rightarrow \infty }\left|\sum_{\vec{m}\le\vec{n}_{\max}}\sum_{\vec{n}\le\vec{n}_{\max}}\beta_{\vec{n}}A(\omega,j(k))_{\vec{n},\vec{m}}u_{\vec{m}}(a,b)\right|\in (0,\infty ].
\]
}}
\end{lem}
\selectlanguage{english}%
\begin{proof}
First recall that the matrix $\left(A(\omega,k)_{\vec{n},\vec{m}}\right)_{\vec{n}\le\vec{n}_{\max},\vec{m}\le\vec{n}_{\max}}$
is invertible, since it is a triangular matrix with all its
diagonal elements equal to one (see Lemma~\ref{lem:idnm}). Since $\left(\beta_{\vec{n}}\right)_{\vec{n}\le\vec{n}_{\max}}\neq0$
we conclude that, for all $k\in\N$, 
\[
\lambda(k):=\left(\lambda_{\vec{m}}(k)\right)_{\vec{m}\le\vec{n}_{\max}}:=\left(\sum_{\vec{n}\le\vec{n}_{\max}}\beta_{\vec{n}}
A(\omega,j(k))_{\vec{n},\vec{m}}\right)_{\vec{m}\le\vec{n}_{\max}}\neq 0.
\]
Let $\epsilon >0$ and 
now suppose by way of contradiction that $\eta=0$ for all $a,b\in B(x_{0},\epsilon)\setminus 
\{ x_{0}\} $ 
with $a\neq b.$ Then we have for
all $a,b\in B(x_{0},\epsilon)\setminus\left\{ x_{0}\right\} $, 
\begin{equation}
\lim_{k\rightarrow \infty }\sum_{\vec{m}\le\vec{n}_{\max}}\lambda_{\vec{m}}(k)u_{\vec{m}}(a,b)=0.\label{contradiction-eq}
\end{equation}
Since $\lambda(k)\neq0$ we may define $\lambda_{0,\vec{m}}(k):=\lambda_{\vec{m}}(k)/\Vert(\lambda_{\vec{p}}(k))_{\vec{p}\le\vec{n}_{\max}}\Vert$.
Here, for every $\gamma =(\gamma _{\vec{p}})_{\vec{p}\leq \vec{n}_{\max}}$, 
we set $\| \gamma \| =\| (\gamma _{\vec{p}})_{\vec{p}\leq \vec{n}_{\max}}\| =
\sqrt{\sum _{\vec{p}\leq \vec{n}_{\max}}|\gamma _{\vec{p}}|^{2}}.$  
By passing to a subsequence $(j (k_{\ell}))_{\ell \in\N}$ of $(j(k))_{k\in \N }$ we may assume that
$\lambda_{\vec{m}}:=\lim_{\ell \rightarrow \infty }\lambda_{0,\vec{m}}(k_{\ell })\in \C $ exists for
each $\vec{m}\le\vec{n}_{\max}$. Put 
$\lambda:=\left(\lambda_{\vec{m}}\right)_{\vec{m}\le\vec{n}_{\max}}$
and observe that $\Vert\lambda\Vert=1$. 
Let $\vec{r}$ be a maximal element in $\{ \vec{n}:\vec{n}\leq \vec{n}_{\max}, \beta _{\vec{n}}\neq 
0\}$ with respect to $\leq .$ Then by Lemma~\ref{lem:idnm}, 
$$\sum _{\vec{n}\leq \vec{n}_{\max}}\beta _{\vec{n}}\cdot 
A(\omega , j(k_{\ell}))_{\vec{n}.\vec{r}}
=\beta _{\vec{r}}A(\omega , j(k_{\ell}))_{\vec{r},\vec{r}}=\beta _{\vec{r}},  \mbox{ \ for every }\ell \in \N.$$  
Thus, $\| (\lambda _{\vec{p}}(k_{\ell}))_{\vec{p}\leq \vec{n}_{\max}}\| 
\geq |\beta _{\vec{r}}|>0$ for every $\ell \in \N.$  
Hence, it follows from (\ref{contradiction-eq})
that 
\[
\sum_{\vec{m}\le\vec{n}_{\max}}\lambda_{\vec{m}}u_{\vec{m}}(a,b)=0,  
\mbox{\ for all }a,b\in B(x_{0},\epsilon )\setminus \{ x_{0}\}, 
\]
which yields $\lambda=0$ by Proposition \ref{prop:fullrank}. This
is the desired contradiction which completes the proof of the lemma. \end{proof}
\selectlanguage{british}%
\begin{lem}
\label{lem:upperbound}Let $\sum_{\vec{n}}\beta_{\vec{n}}C_{\vec{n}}\in\T$
be non-trivial. Let $\omega\in I^{\N}$, $z\in J_{\omega}$ and $\vec{n}\in\N_{0}^{s}$.
Then 
\[
\Hol(\sum_{\vec{n}}\beta_{\vec{n}}C_{\vec{n}},z)\le
\liminf_{k\rightarrow \infty }\frac{S_{k}\tilde{\psi}\left(\omega,z\right)}{S_{k}\tilde{\varphi}\left(\omega,z\right)}.
\]
 \end{lem}
 \vspace{-5mm} 
\begin{proof}
Let $\alpha =\liminf_{k\rightarrow \infty }
(S_{k}\tilde{\psi} (\omega ,z))/(S_{k}\tilde{\varphi }(\omega ,z)).$ We may assume that there exists a sequence $(j(k))_{k\in\N}$ tending
to infinity such that 
\[
\lim_{k\rightarrow \infty }\frac{S_{j(k)}\tilde{\psi}(\omega,z)}{S_{j(k)}\tilde{\varphi}(\omega,z)}=\alpha.
\]
Moreover, since $f_{\omega |_{j(k)}}(z)\in J_{\sigma^{j(k)}(\omega)}\subset J(G)$
and $J(G)$ is compact, we may assume that $x_{0}:=\lim_{k}f_{\omega |_{j(k)}}(z)$
exists.
Let $V,R$ be as in 
the  proof of Lemma \ref{lem:lowerbound}. 
Then for each $p\in\N$, \foreignlanguage{english}{there
exists a holomorphic branch $\phi_{p}:B\bigl(f_{\omega |_{p}}\left(z\right),R\bigr)\rightarrow\Chat$
of $f_{\omega |_{p}}^{-1}$ such that 
$f_{\omega |_{p}}\left(\phi_{p}\left(y\right)\right)=y$
for $y\in B\bigl(f_{\omega |_{p}}\left(z\right),R\bigr)$ and $\phi_{p}
\bigl(f_{\omega |_{p}}\left(z\right)\bigr)=z$. 
Since $G(V)\subset V$, by taking $R$ so small, we may assume that 
$f_{\omega |_{q}}(\phi _{p}(B(f_{\omega |_{p}}(z),R)))\subset B(f_{\omega |_{q}}(z),\epsilon _{0})$ 
for all $p,q\in \N \cup \{ 0\} $ with $0\leq q\leq p$, where 
$\epsilon _{0}$ is the number given in Lemma~\ref{lem:e0}. 
Let $\epsilon>0$. By Lemma \ref{lem:eta-keylemma} there exist $a,b\in B(x_{0},\epsilon)\setminus\left\{ x_{0}\right\} $
such that 
\[
\eta:=\limsup_{k\rightarrow \infty }\left|\sum_{\vec{m}\le\vec{n}_{\max}}\sum_{\vec{n}\le\vec{n}_{\max}}\beta_{\vec{n}}A(\omega,j(k))_{\vec{n},\vec{m}}u_{\vec{m}}(a,b)\right|\in (0, \infty ].
\]
After passing to a subsequence of $(j(k))_{k\in \N }$ if necessary, 
we may assume that 
\[
\eta=\lim_{k\rightarrow \infty }\left|\sum_{\vec{m}\le\vec{n}_{\max}}\sum_{\vec{n}\le\vec{n}_{\max}}\beta_{\vec{n}}A(\omega,j(k))_{\vec{n},\vec{m}}u_{\vec{m}}(a,b)\right|\in (0, \infty ].
\]
For sufficiently large $k\in\N$ and $\epsilon$ small, we may assume
that $a,b\in B\bigl(f_{\omega |_{j(k)}}\left(z\right),R\bigr)$. We
set $y_{k}:=\phi_{j(k)}(a)$ and $z_{k}:=\phi_{j(k)}(b)$. 
Let $\vec{n}_{\max}\in \N_{0}^{s}$  such that 
if $\vec{n}\in \N_{0}^{s}, \beta _{\vec{n}}\neq 0$ then 
% \vec{n}\mid\beta_{\vec{n}}\neq0\right\} $.
$\vec{n}\leq \vec{n}_{\max}.$ 
By Lemma \ref{lem:matrix-formulas} we have 
\begin{eqnarray*}
C(y_{k})-C(z_{k}) & = & \sum_{\vec{n}\le\vec{n}_{\max}}\beta_{\vec{n}}\left(C_{\vec{n}}(y_{k})-C_{\vec{n}}(z_{k})\right)\\
 & = & p_{\omega |_{j(k)}}\sum_{\vec{n}\le\vec{n}_{\max}}\beta_{\vec{n}}\sum_{\vec{m}\in\N_{0}^{s}:\vec{m}\le\vec{n}}A(\omega,j(k))_{\vec{n},\vec{m}}\cdot u_{\vec{m}}(a,b)\\
 & = & p_{\omega |_{j(k)}}\sum_{\vec{m}\le\vec{n}_{\max}}\left(\sum_{\vec{n}\le\vec{n}_{\max}}\beta_{\vec{n}}
 A(\omega,j(k))_{\vec{n},\vec{m}}\right)u_{\vec{m}}(a,b).
\end{eqnarray*}
Let $\eta _{0}\in (0,\eta ).$ Since $S_{j(k)}\tilde{\varphi }<0$ for all large $k$ 
(see the proof 
of Lemma~\ref{lem:lowerbound}),  
it follows that
\[
\liminf_{k\rightarrow \infty }\frac{\log\left|C(y_{k})-C(z_{k})\right|}{S_{j(k)}\tilde{\varphi}\left(\omega,z\right)}\le\liminf_{k\rightarrow \infty }\frac{S_{j(k)}\tilde{\psi}\left(\omega,z\right)+\log\eta_{0}}{S_{j(k)}\tilde{\varphi}\left(\omega,z\right)}=\alpha.
\]
By Koebe's distortion theorem 
%and the definition of $\tilde{\varphi}$
%it follows that 
we have 
\begin{equation}
\liminf_{k\rightarrow \infty }\frac{\log\left|C(y_{k})-C(z_{k})\right|}{\log\left(d(y_{k},z_{k})\right)}\le\alpha.\label{eq:exponent-is-alpha-2}
\end{equation}
Finally, we show that $\Hol(C,z)\le\alpha$. To prove this, we show
that $\Hol(C,z)\le\beta$ for every $\beta>\alpha$. Suppose by way
of contradiction that $\Hol(C,z)>\beta$. By the triangle inequality
we have
\begin{equation}
\left|C(y_{k})-C(z)\right|\ge\left|\left|C(y_{k})-C(z_{k})\right|-\left|C(z_{k})-C(z)\right|\right|.\label{eq:triangle}
\end{equation}
By Koebe's distortion theorem we have that $d(y_{k},z)\asymp d(y_{k},z_{k})\asymp d(z_{k},z)$
as $k$ tends to infinity. Consequently, by combining with (\ref{eq:triangle}),
we see that there exists a constant $K>1$ such that 
\begin{equation}
\frac{\left|C(y_{k})-C(z)\right|}{d(y_{k},z)^{\beta}}\ge K^{-1}\frac{\left|C(y_{k})-C(z_{k})\right|}{d(y_{k},z_{k})^{\beta}}-K\frac{\left|C(z_{k})-C(z)\right|}{d(z_{k},z)^{\beta}}.\label{eq:hoelder-triangle}
\end{equation}
Our assumption $\Hol(C,z)>\beta$ implies that 
\[
\lim_{k\rightarrow \infty }\left|C(z_{k})-C(z)\right|\big/d(z_{k},z)^{\beta}=0\quad\mbox{and}\quad
\lim_{k\rightarrow \infty }\left|C(y_{k})-C(z)\right|\big/d(y_{k},z)^{\beta}=0.
\]
Moreover, by (\ref{eq:exponent-is-alpha-2}) and our assumption that
$\beta>\alpha$ we have $\limsup_{k}\left|C(y_{k})-C(z_{k})\right|\big/d(y_{k},z_{k})^{\beta}=\infty$.
Now (\ref{eq:hoelder-triangle}) gives the desired contradiction and
 finishes the proof of the lemma. }
\end{proof}
%\vspace{-10mm} 
We conclude that Theorem~\ref{main-thm} follows from combining Lemmas~\ref{lem:lowerbound} 
and \ref{lem:upperbound}. 
%\vspace{2mm} 
%e\end{proof} 
%\vspace{-5mm} 
\section{Proofs of Theorems  \ref{thm:multifractal} and \ref{thm:c0holder}}
\label{section:pfmulti}
%In this section we give the proofs of Theorems \ref{thm:multifractal} and  \ref{thm:c0holder}. 
%\vspace{-2mm} 
In this section, we give the proofs of Theorems~\ref{thm:multifractal} and \ref{thm:c0holder}.   
The proof of Theorem \ref{thm:multifractal} will follow from the detailed version  Theorem \ref{thm:mf-for-hoelderexponent} stated below. For $C\in \T$ and $z\in\Chat$ we define
\[
\M_{*}\left(C,z\right):=\liminf_{r\rightarrow0}\frac{\log\M\left(C,z,r\right)}{\log r},
%,\;\M^{*}\left(C,z\right):=\limsup_{r\rightarrow0}\frac{\log\M\left(C,z,r\right)}{\log r}\;\mbox{and}\;\M\left(C,z\right):=\lim_{r\rightarrow0}\frac{\log\M\left(C,z,r\right)}{\log r},
\]
where $\M\left(C,z,r\right)$ is for $r>0$ given by 
\[
\M\left(C,z,r\right):=\sup_{y\in B\left(z,r\right)}\left|C\left(y\right)-C\left(z\right)\right|.
\]
Moreover, we define for each $\alpha\in\R$ the corresponding level
sets
\[
R_{*}\left(C,\alpha\right):=\left\{ y\in\Chat:\M_{*}\left(C,y\right)=\alpha\right\}. 
%,\; R^{*}\left(C,\alpha\right):=\left\{ y\in\Chat:\M^{*}\left(C,y\right)=\alpha\right\} 
\]
%\vspace{-2mm} 
%and 
%\vspace{-2mm} 
%\[
%R\left(C,\alpha\right):=\left\{ y\in\Chat:\M\left(C,y\right)=\alpha%\right\} .
%\]
Also, we define the dynamically defined level sets $\mathcal{F}\left(\alpha\right)$,
 which are for $\alpha\in\R$ given by
\[
\mathcal{F}\left(\alpha\right):=\pi\left(\tilde{\mathcal{F}}\left(\alpha\right)\right),\mbox{ where }\tilde{\mathcal{F}}\left(\alpha\right):=\biggl\{ (\omega ,x)\in J\left(\tilde{f}\right):\lim_{n\rightarrow\infty}\frac{S_{n}\tilde{\psi}\left(\omega, x\right)}{S_{n}\tilde{\varphi}\left(\omega, x\right)}=\alpha\biggr\}.
\]
Moreover, for $\alpha \in \R $ we set 
$${\mathcal F}'(\alpha ):=\pi (\tilde{{\mathcal F}}'(\alpha )), 
\mbox{ where } 
%\tilde{{\mathcal F}}'(\alpha ):=\{ 
%x\in J(\tilde{f})
\tilde{\mathcal{F}}'\left(\alpha\right):=\biggl\{ (\omega ,x)\in J\left(\tilde{f}\right):\liminf_{n\rightarrow\infty}\frac{S_{n}\tilde{\psi}\left(\omega, x\right)}{S_{n}\tilde{\varphi}\left(\omega, x\right)}=\alpha\biggr\}.
$$
%The proof is based on the  thermodynamic formalism and multifractal formalism for expanding rational %semigroups (see  \cite[Section 3]{JS13b}).  We proceed by introducing the  necessary definitions. 
The free energy function  is defined by the
unique function $t:\R\rightarrow\R$ such that $\mathcal{P}\bigl(\beta\tilde{\psi}+t\left(\beta\right)\tilde{\varphi},\tilde{f}\bigr)=0$
for each $\beta\in\R$, where $\mathcal{P}\left(\cdot,\tilde{f}\right)$
denotes the topological pressure with respect to the dynamical system $(J(\tilde{f}),\tilde{f})$ (cf.   \cite{MR648108}).  The number $t\left(0\right)$ is also referred to as the critical
exponent $\delta$ of the rational semigroup $G=\langle f_1,\dots ,f_{s+1} \rangle$
(see \cite{MR2153926}).  
Note that under assumptions (1)(2) of our paper, we have 
$$\delta =\dim _{H}(J(G))$$ (\cite{MR1625124, MR2153926}). 
The convex conjugate of $t$ (\cite[Section 12]{rockafellar-convexanalysisMR0274683})
is given by 
\[
t^{*}:\R\rightarrow\R\cup\left\{ \infty\right\} ,\quad t^{*}\left(c\right):=\sup_{\beta\in\R}\left\{ \beta c-t\left(\beta\right)\right\} ,\quad c\in\R.
\]
\vspace{-2mm} 
We now present the theorem. 
\begin{thm}
\label{thm:mf-for-hoelderexponent} Every  non-trivial  $C\in \T$ satisfies all of the following. 
\vspace{-1mm} 
\begin{enumerate}
%\item 
%$C$ is $a$-H\"older continuous on $\Chat$ for every
%$a<\alpha_{-}$. Moreover, $C_{0}$ is $\alpha_{-}$-H\"older continuous on $\Chat$. 
%\selectlanguage{english}%
\item \textup{\emph{We have $\alpha_{+}=\sup\left\{ \alpha\in\R:R_{*}\left(C,\alpha\right)\neq\emptyset\right\} $
and $\alpha_{-}=\inf\left\{ \alpha\in\R:R_{*}\left(C,\alpha\right)\neq\emptyset\right\} $.
 }}
Moreover, we have $\alpha _{-}=\inf \{ \alpha \in \R : {\mathcal F}(\alpha )\neq \emptyset \} 
$ and $\alpha _{+}=\sup \{ \alpha \in \R : {\mathcal F}(\alpha )\neq \emptyset \} .$ 
\selectlanguage{british}%
\item Let  $\alpha_{0}:=-t'(0)$.
If $\alpha_{-}<\alpha_{+}$ then for each $\alpha\in\left(\alpha_{-},\alpha_{+}\right)$, 
we have ${\mathcal F}(\alpha )\subset {\mathcal F}'(\alpha )=H(C, \alpha)$ and  
\begin{align*}
\dim_{H}\left(\mathcal{F}\left(\alpha\right)\right) & =\dim_{H}\left(R_{*}\left(C,\alpha\right)\right)=\dim_{H}\left(H\left(C,\alpha\right)\right)=-t^{*}\left(-\alpha\right)>0.
\end{align*}
Moreover, $v(\alpha ):=-t^{*}\left(-\alpha \right)$ is a real analytic
strictly concave positive function on $\left(\alpha_{-},\alpha_{+}\right)$
with maximum value $\dim _{H}(J(G))=\delta=-t^{*}\left(-\alpha_{0}\right)>0$.  
Also, $v''<0$ on $(\alpha _{-}, \alpha _{+}).$ 
\item 

\begin{enumerate}
\item We have $\alpha_{-}=\alpha_{+}$ if and only if there exist an automorphism
$\varphi\in\Aut\bigl(\Chat\bigr)$ and $a_1,\dots ,a_{s+1}\in\C$
such that 
\vspace{-3mm} 
\[
\varphi\circ f_{i}\circ\varphi^{-1}\left(z\right)=
a_{i}z^{\pm\deg\left(f_{i}\right)}\quad
\mbox{and}\quad\log\deg\left(f_{i}\right)=-\left(\gamma/\delta\right)\log p_{i},
\]
where   $\gamma$ denotes  the unique number such that $\mathcal{P}\left(\gamma\tilde{\psi},\tilde{f}\right)=0$. 

\item If $\alpha_{-}=\alpha_{+}$ then we have 
\vspace{-1mm} 
\[
\mathcal{F}\left(\alpha_{0}\right) =R_{*}\left(C,\alpha_{0}\right)=H\left(C,\alpha_{0}\right)=J(G),
\]
\vspace{-3mm} 
 and for all $\alpha\in \R $ with $\alpha \neq\alpha_{0}$  we have 
 \vspace{1mm} 
\[
\mathcal{F}\left(\alpha \right)=R_{*}\left(C,\alpha\right)=H\left(C,\alpha\right)=\emptyset.
\]

\end{enumerate}
\end{enumerate}
\end{thm}
\vspace{-7mm} 
\begin{proof}
%We first prove (1). 
The  assertions follow by combining Theorem \ref{main-thm} with \cite[Theorem 5.3, Proposition 3.9, Lemma 5.1]{JS13b}, which  completes the proof of the theorem.
\end{proof}
\vspace{-2mm} 
We now give the proof of Theorem~\ref{thm:c0holder}.
\vspace{-3mm} 
\begin{proof}
By \cite{MR1625124} and \cite[Remark 5]{MR2153926},  assumption (1) of this paper 
is equivalent to 
the property that the associated skew product map 
$\tilde{f}$ 
is expanding in the sense of \cite{MR2153926} and  \cite{JS13b}.   
By \cite{MR2153926} or \cite{JS13b} again, we have that 
$\left(J\left(\tilde{f}\right),\tilde{f}\right)$ is a topological transitive expanding
dynamical system with compact state space.  
Thus, there exists a surjective $\Hol$der continuous morphism from an irreducible Markov
shift over a finite alphabet. The Markov shift $(\Sigma_{A},\sigma)$ is given
by the shift space 
\[
\Sigma_{A}:=\left\{ \omega=\left(\omega_{i}\right)_{i\in\N}\in V^{\N}\mid A(\omega_{i},\omega_{i+1})=1\mbox{ for all }i\in \N \right\} ,
\]
and the left shift $\sigma:\Sigma_{A}\rightarrow\Sigma_{A}$, where
$V$ is a finite alphabet set and $A\in\left\{ 0,1\right\} ^{V\times V}$
is an irreducible incidence matrix. We endow $\Sigma_{A}$ with the
metric 
\[
d_{\Sigma}(\omega,\tau):=2^{-\sup\left\{ n\ge0\mid\omega_{1}=\tau_{1},\dots,\omega_{n}=\tau_{n}\right\} }.
\]
It is known (see e.g. \cite[proof of Proposition 3.10]{JS13b}) that there exists a surjective H\"older continuous
map $\pi_{\Sigma }:\Sigma_{A}\rightarrow J\left(\tilde{f}\right)$ 
 such that $\tilde{f}\circ\pi_{\Sigma }=\pi_{\Sigma}\circ\sigma$. We
define the potentials $\psi:\Sigma_{A}\rightarrow\R$, $\varphi:\Sigma_{A}\rightarrow\R$
given by 
\[
\psi:=\tilde{\psi}\circ\pi_{\Sigma},\quad\varphi:=\tilde{\varphi}\circ\pi_{\Sigma}.
\]
Note that $\psi$ and $\varphi$ are H\"older continuous. For a function
$g:\Sigma_{A}\rightarrow\R$ we write $S_{n}g:=\sum_{j=0}^{n-1}g\circ\sigma^{j}$.
If $g:\Sigma_{A}\rightarrow\R$ is H\"older continuous, then there
exists a constant $K_{g}\ge1$ such that, for every $k\in\N$, $\omega_{1}\dots\omega_{k}u\in\Sigma_{A}$ and $\omega_{1}\dots\omega_{k}v\in\Sigma_{A}$ we have the bounded distortion property
\[
\left|S_{k}g(\omega_{1}\dots\omega_{k}u)-S_{k}g(\omega_{1}\dots\omega_{k}v)\right|\le K_{g}.
\]
As in the proof of Lemma~\ref{lem:lowerbound}, 
there exist $m\in \N $ and $\theta <0$ such that 
$S_{m}\varphi \leq \theta <0.$ 
Since $A$ is irreducible, there exists $l_{0}\in\N$ such that for
each $k\in\N$ and $\omega\in\Sigma_{A}^{km}$ there exists $m\leq l\le l_{0}$
and $\tau\in\Sigma_{A}^{l}$ such that $\overline{\omega\tau}:=\omega\tau\omega\tau\dots\in\Sigma_{A}$.
Here, $\Sigma _{A}^{n}:=\{ \omega =(\omega _{i})_{i=1}^{n}\in V^{n}: A(\omega _{i},\omega _{i+1})=1, 
i=1,\ldots, n-1\} .$  

By the definition of $\alpha_{-}$ and $\alpha_{+}$ and Theorem \ref{main-thm}
we have 
\begin{equation}
\frac{S_{km+l}\psi(\overline{\omega\tau})}{S_{km+l}\varphi(\overline{\omega\tau})}=\lim_{n\rightarrow\infty}\frac{S_{n}\psi(\overline{\omega\tau})}{S_{n}\varphi(\overline{\omega\tau})}=\lim_{n\rightarrow\infty}\frac{S_n\tilde{\varphi}(\pi_{\Sigma}(\overline{\omega\tau}))}{S_{n}\tilde{\varphi}(\pi_{\Sigma}(\overline{\omega\tau}))}\in\left[\alpha_{-},\alpha_{+}\right].\label{eq:limit on periodic points}
\end{equation}
For each  $x\in J(\tilde{f})$ and  $k\in\N$ there exists $\omega\in\Sigma_{A}$
such that $\pi_{\Sigma}(\omega)=x$. Let $m\leq l\le l_{0}$ and $\tau\in\Sigma_{A}^{l}$
such that $\overline{(\omega_{1}\dots\omega_{km}\tau)}\in\Sigma_{A}$.
Using the bounded distortion property of $\psi$ and $\varphi$ and
(\ref{eq:limit on periodic points}) we obtain that for large $k$, 
\begin{eqnarray*}
\frac{S_{km}\tilde{\psi}(x)}{S_{km}\tilde{\varphi}(x)}=\frac{S_{km}\psi(\omega)}{S_{km}\varphi(\omega)} 
& \ge & \frac{-S_{km}\psi(\overline{\omega_{1}\dots\omega_{km}\tau_{1}\dots\tau_{l}})-K_{\psi}}{-S_{km}\varphi(\overline{\omega_{1}\dots\omega_{km}\tau_{1}\dots\tau_{l}})+K_{\varphi}}\\ 
& \ge & \frac{-S_{km+l}\psi(\overline{\omega_{1}\dots\omega_{km}\tau_{1}\dots\tau_{l}})-K_{\psi}-l_{0}\max\left|\psi\right|}{-S_{km+l}\varphi(\overline{\omega_{1}\dots\omega_{km}\tau_{1}\dots\tau_{l}})+K_{\varphi}}\\
 & \ge & \frac{S_{km+l}\psi(\overline{\omega_{1}\dots\omega_{km}\tau_{1}\dots\tau_{l}})}{S_{km+l}\varphi(\overline{\omega_{1}\dots\omega_{km}\tau_{1}\dots\tau_{l}})}\cdot\delta(k)\ge\alpha_{-}\cdot\delta(k),
\end{eqnarray*}
where we have set 
\[
\delta(k):=\left(1-\frac{K_{\psi}+l_{0}\max\left|\psi\right|}{k\min\left|S_{m}\psi\right|}\right)\Big/\left(1+\frac{K_{\varphi}}{k\min\left|S_{m}\varphi\right|}\right).
\]
For  $k\in\N$ we have $r_{k}:=\e^{S_{km}\tilde{\varphi}(x)}<1$. Consequently, we have
\begin{eqnarray*}
\e^{S_{km}\tilde{\psi}(x)}=\left(\e^{S_{km}\tilde{\varphi}(x)}\right)^{\frac{S_{km}\tilde{\psi}(x)}{S_{km}\tilde{\varphi}(x)}} \le r_k^{\alpha_{-} \cdot \delta(k)}.
\end{eqnarray*}
Since $\delta(k)=1+O(k^{-1})$, as $k\rightarrow\infty$, and 
% we have
%\[
%r_{k}^{\delta (k)}=r_{k}^{1+O(k^{-1})}=r_{k}\e^{\log(r_{k})O(k^{-1})},\quad\mbox{as }k\rightarrow\infty.
%\]
 ${\log(r_{k})}/{k}={S_{km}\tilde{\varphi}}/{k}$ is a  bounded sequence, we have thus shown that there exists a constant $K=K(\tilde{\varphi},\tilde{\psi})\ge1$ such that for all $x\in J(\tilde{f})$ and for all $k\in \N$, 
 \begin{equation} \label{equ:birkhoffquotients}
 \e^{S_{km}\tilde{\psi}(x)}\le K^{\alpha_-} r_k^{\alpha_-}= K^{\alpha_-} \e^{\alpha_-\cdot S_{km}\tilde{\varphi}(x)}.
 \end{equation}
%we conclude that there exists a constant $K=K(\tilde{\varphi},\tilde{\psi})\ge1$ such that 
%$r_{k}^{\delta(k)}\le Kr_{k}$. We have thus verified the following lemma.
%\begin{lem} \label{lemma-birkhoffquotients} There exists a constant $K=K(\tilde{\varphi},\tilde{\psi})\ge1$ such that for all $x\in J(\tilde{f})$ and for all $k\in \N$, 
%\[
%\e^{S_{km}\tilde{\psi}(x)}\le K^{\alpha_-} r_k^{\alpha_-}= K^{\alpha_-} \e^{\alpha_-\cdot S_{km}\tilde{\varphi}(x)}.
%\]
%\end{lem}
% Note that $\delta(k)=1+O(k^{-1})$ as $k\rightarrow\infty$. 
Let  $C\in\T$ be non-trivial and let $z\in J(G)$ with  $z=\pi (x)$ for some $x\in J(\tilde{f})$. 
We use the notations of the proof of 
Lemma~\ref{lem:lowerbound}. 
By (\ref{eq:upperestimate}) in the proof
of Lemma \ref{lem:lowerbound} there exists $p\in\N$ and $K_{1}$
(depending only on $C$) such that, for all $k\in\N$, 
\begin{eqnarray*}
\sup_{y\in B_{km}}\left|C(y)-C(z)\right| & \le & K_{1}(km)^{p}\e^{S_{km}\tilde{\psi}(x)}.
%=K_{1}(km)^{p}\left(\e^{S_{km}\tilde{\varphi}(x)}\right)^{\frac{S_{km}\tilde{\psi}(x)}{S_{km}\tilde{\varphi}(x)}}.
\end{eqnarray*}
%Consequently,  we have   
%\[
%\sup_{y\in B_{km}}\left|C(y)-C(z)\right|\le K_{1}(km)^{p}r_{k}^{\alpha_{-}\cdot\delta(k)}.
%\]
%Since $\delta(k)=1+O(k^{-1})$, as $k\rightarrow\infty$, we have
%\[
%r_{k}^{\delta (k)}=r_{k}^{1+O(k^{-1})}=r_{k}\e^{\log(r_{k})O(k^{-1})},\quad\mbox{as }k\rightarrow\infty.
%\]
%Since ${\log(r_{k})}/{k}={S_{km}\tilde{\varphi}}/{k}$ is a  bounded sequence, we conclude that there exists a constant $K_{2}\ge1$ such that 
%$r_{k}^{\delta(k)}\le K_{2}r_{k}$. 
Combining with \ref{equ:birkhoffquotients} we obtain 
\[
\sup_{y\in B_{km}}\left|C(y)-C(z)\right|\le K_{1}(km)^{p} K^{\alpha_-} r_k^{\alpha_-} .
\]
Thus, since $\alpha _{-}>0$,  
for every  $\epsilon>0$ we have 
\[
\sup_{y\in B_{km}}\left|C(y)-C(z)\right|\le K_{1}K^{\alpha _{-}}(km)^{p}r_{k}^{\epsilon\alpha_{-}}r_{k}^{\alpha_{-}\cdot(1-\epsilon)}.
\]
Also, $k^{p}r_{k}^{\epsilon\alpha_{-}}\rightarrow0$ as $k\rightarrow\infty$. 
Combining the above with the Koebe distortion theorem, 
we obtain that $C$ is $\alpha_{-}\cdot(1-\epsilon)$-H\"older
continuous on $\Chat$. Note that it suffices to prove the H\"older
continuity  for $z\in J(G)$. To see this, let $z,y\in F(G)$.  Then either $y,z$ belong to the same connected component of $F(G)$, and thus $C(y)=C(z)$ by Corollary \ref{cor:varies}, or there exists $u\in J(G)$ between $y$ and $z$ and the desired H\"older continuity follows from the triangle inequality. Finally, if $C=C_{0}$ then the previous estimates hold with $p=0$
(and hence $\epsilon=0$), which implies that $C_{0}$ is $\alpha_{-}$-
$\Hol$der continuous on $\Chat .$ 
The proof of Theorem~\ref{thm:c0holder} is complete.
\end{proof} 

\section{Proof of Theorem \ref{thm:less-than-one}}
\label{section:pfl-t-o}
\vspace{-3mm} 
In this section, we give the proof of Theorem~\ref{thm:less-than-one}.
\vspace{-2mm} 
\begin{proof}
Suppose that $\alpha_{-} \ge 1$. Then by Theorem \ref{thm:c0holder}  we have that 
$C_{0}$ is a Lipschitz function on $\Chat$. Let $K$ be a minimal set of $G$ with $K\neq L$. 
By conjugating $G$ by a M\"obius transformation, we may assume that  $J(G)$ is a subset of $\C.$
Let $ABCD$ be a rectangle such that $AB$ is included in
a connected component $ U_{L}$ of $F(G)$ with $U_{L}\cap L\neq \emptyset$,
and $CD$ is included in a connected component $U_{K}$ of $F(G)$
with $U_{K}\cap K\neq \emptyset .$
Since the $2$-dimensional Lebesgue measure of $J(G)$ is zero
(actually $\dim _{H}(J(G))<2$), Fubini's theorem
implies that there exists a segment $S$ in $ABCD$ which joins
$AB$ and $CD$ such that the $1$-dimensional Lebesgue measure of $S\cap J(G)$ is zero.
Let us consider $E=C_{0}|_{S}$. Identify $S$ with $[a,b]\subset \R$ 
such that $a$ corresponds to a point in $AB\subset U_{L}$ 
and b corresponds to a point in $CD\subset U_{K}.$  Note that by the definition of $C_0$ we have that $E(a)=1$ and $E(b)=0$.
Since $E$ is  Lipschitz, it is almost everywhere differentiable
on $S$ with respect to the $1$-dimensional Lebesgue measure on $S$ and we have
$E(x)=\int _{a}^{x}E'(t) dt.$ But
$E$ is locally constant on $S\cap F(G)$, and since the 
$1$-dimensional Lebesgue measure of $S\cap J(G)$ is zero,
we have $E'(x)=0$ almost everywhere on $S$, which implies that 
$E$ is constant on $S$. This is the desired contradiction which completes the proof 
of the result $\alpha _{-}<1.$ 

We now let $\alpha \in (\alpha _{-},\min\{ \alpha _{+}, 1\}).$ 
Then Theorems~\ref{main-thm} and \ref{thm:mf-for-hoelderexponent} 
imply that there exists a Borel subset $A_{0}$ of $J(G)$ with 
$\dim _{H}(A_{0})>0$ such that 
for every $x\in A_{0}$ and for every non-trivial $C\in {\mathcal C}$, 
we have $\Hol(C,x)=\alpha .$ 
Let $A=\cup _{g\in G}g^{-1}(A_{0}).$ 
Then $\overline{A}=J(G)$ (\cite[Lemma 3.2]{MR1397693}) and  Theorem \ref{main-thm}  implies that $A$ has the desired property. 
\end{proof}

{\bf Acknowledgements.}  
The authors would like to thank Rich Stankewitz for valuable comments. 
The research of the first author was partially supported by 
the research fellowship JA 2145/1-1 of the German Research Foundation (DFG), by the JSPS postdoctoral research fellowship for foreign researchers (P14321) and by the JSPS Kakenhi 90741869. 
The research of the second author was partially supported by 
JSPS KAKENHI 24540211,  15K04899. 

\def\cprime{$'$}
\providecommand{\bysame}{\leavevmode\hbox to3em{\hrulefill}\thinspace}
\providecommand{\MR}{\relax\ifhmode\unskip\space\fi MR }
\providecommand{\MRhref}[2]{%
  \href{http://www.ams.org/mathscinet-getitem?mr=#1}{#2}
}
\providecommand{\href}[2]{#2}

\end{document}